\documentclass[a4paper,11pt]{amsart}
\usepackage[left=2.7cm,right=2.7cm,top=3.5cm,bottom=3cm]{geometry}

\usepackage{amsthm,amssymb,amsmath,amsfonts,mathrsfs,amscd}
\usepackage[latin1]{inputenc}
\usepackage[all]{xy}
\usepackage{latexsym}
\usepackage{longtable}
\newfont{\cyr}{wncyr10 scaled 1100}
\newfont{\cyrr}{wncyr9 scaled 1000}

\theoremstyle{plain}
\newtheorem{theorem}{Theorem}[section]
\newtheorem{proposition}[theorem]{Proposition}
\newtheorem{lemma}[theorem]{Lemma}
\newtheorem{corollary}[theorem]{Corollary}

\theoremstyle{definition}

\newtheorem{definition}[theorem]{Definition}
\newtheorem{assumption}[theorem]{Assumption}

\theoremstyle{remark}
\newtheorem{remark}[theorem]{Remark}

\newcommand{\Q}{\mathbb Q}
\newcommand{\N}{\mathbb N}
\newcommand{\Z}{\mathbb Z}
\newcommand{\R}{\mathbb R}
\newcommand{\C}{\mathbb C}

\DeclareMathOperator{\Aut}{Aut}
\DeclareMathOperator{\Frob}{Frob}

\DeclareMathOperator{\Hom}{Hom}

\DeclareMathOperator{\Gal}{Gal}
\DeclareMathOperator{\GL}{GL}

\DeclareMathOperator{\Sel}{Sel}
\DeclareMathOperator{\M}{M}

\newcommand{\res}{\mathrm{res}}
\newcommand{\cores}{\mathrm{cores}}

\newcommand{\tr}{\mathrm{tr}}
\newcommand{\ord}{\mathrm{ord}}

\newcommand{\Sha}{\mbox{\cyr{X}}}

\usepackage[usenames]{color}
\definecolor{Indigo}{rgb}{0.2,0.1,0.7}
\definecolor{Violet}{rgb}{0.5,0.1,0.7}
\definecolor{White}{rgb}{1,1,1}
\definecolor{Green}{rgb}{0.1,0.9,0.2}

\newcommand{\longmono}{\mbox{\;$\lhook\joinrel\longrightarrow$\;}}

\newcommand{\longepi}{\mbox{\;$\relbar\joinrel\twoheadrightarrow$\;}}

\newcommand{\dirlim}{\mathop{\varinjlim}\limits}
\newcommand{\invlim}{\mathop{\varprojlim}\limits}

\newcommand{\p}{\mathfrak p}

\newcommand{\cO}{{\mathcal O}}
\newcommand{\E}{{\mathcal E}}
\newcommand{\HH}{{\mathcal H}}

\include{thebibliography}

\begin{document}

\title[Plus/minus Heegner points and supersingular Iwasawa theory]{Plus/minus Heegner points and Iwasawa theory\\of elliptic curves at supersingular primes} 
%\today
\author{Matteo Longo and Stefano Vigni}

\thanks{The two authors are supported by PRIN 2010--11 ``Arithmetic Algebraic Geometry and Number Theory''. The first author is also supported by PRAT 2013 ``Arithmetic of Varieties over Number Fields''; the second author is also supported by PRA 2013 ``Geometria Algebrica e Teoria dei Numeri''.}

\begin{abstract}
Let $E$ be an elliptic curve over $\Q$ and let $p\geq5$ be a prime of good supersingular reduction for $E$. Let $K$ be an imaginary quadratic field satisfying a modified ``Heegner hypothesis'' in which $p$ splits, write $K_\infty$ for the anticyclotomic $\Z_p$-extension of $K$ and let $\Lambda$ denote the Iwasawa algebra of $K_\infty/K$. By extending to the supersingular case the $\Lambda$-adic Kolyvagin method originally developed by Bertolini in the ordinary setting, we prove that Kobayashi's plus/minus $p$-primary Selmer groups of $E$ over $K_\infty$ have corank $1$ over $\Lambda$. As an application, when all the primes dividing the conductor of $E$ split in $K$, we combine our main theorem with results of \c{C}iperiani and of Iovita--Pollack and obtain a ``big O'' formula for the $\Z_p$-corank of the $p$-primary Selmer groups of $E$ over the finite layers of $K_\infty/K$ that represents the supersingular counterpart of a well-known result for ordinary primes.\end{abstract}

\address{Dipartimento di Matematica, Universit\`a di Padova, Via Trieste 63, 35121 Padova, Italy}
\email{mlongo@math.unipd.it}
\address{Dipartimento di Matematica, Universit\`a di Genova, Via Dodecaneso 35, 16146 Genova, Italy}
\email{vigni@dima.unige.it}



\maketitle

\maketitle

%\tableofcontents

\section{Introduction} 

Let $E$ be an elliptic curve over $\Q$ of conductor $N>3$. By modularity, $E$ is associated with a normalized newform $f=f_E$ of weight $2$ for $\Gamma_0(N)$, whose $q$-expansion will be denoted by
\[ f(q)=\sum_{n\geq1}a_nq^n,\qquad a_n\in\Z. \]
Let $K$ be an imaginary quadratic field in which all primes dividing $N$ split (i.e., $K$ satisfies the so-called ``Heegner hypothesis'' relative to $N$) and let $p\geq5$ be a prime of good reduction for $E$ that is unramified in $K$. Write $K_\infty/K$ for the anticyclotomic $\Z_p$-extension of $K$, set $G_\infty:=\Gal(K_\infty/K)\simeq\Z_p$ and define $\Lambda:=\Z_p[\![G_\infty]\!]$ to be the Iwasawa algebra of $G_\infty$. Under some technical assumptions, Bertolini showed in \cite{Ber1} that if the reduction of $E$ at $p$ is ordinary then the Pontryagin dual of the $p$-primary Selmer group of $E$ over $K_\infty$ has rank $1$ over $\Lambda$ and is generated by Heegner points. In this paper we prove similar results for Pontryagin duals of restricted (plus/minus) Selmer groups \emph{\`a la} Kobayashi in the supersingular case.

Set $G_\Q:=\Gal(\bar\Q/\Q)$ and let  
\[ \rho_{E,p}:G_\Q\longrightarrow\Aut(T_p(E))\simeq\GL_2(\Z_p) \] 
denote the Galois representation on the $p$-adic Tate module $T_p(E)\simeq\Z_p^2$ of $E$. Assume that $E$ has no complex multiplication and fix once and for all a prime number $p$ for which the following conditions hold.

\begin{assumption} \label{ass} 
\begin{enumerate}
\item $p\geq5$ is a prime of good supersingular reduction for $E$;
\item $\rho_{E,p}$ is surjective.
\end{enumerate}
\end{assumption}

Thanks to Elkies's result on the infinitude of supersingular primes for elliptic curves over $\Q$ (\cite{Elkies})  and Serre's ``open image'' theorem (\cite{Serre-Inv}), we know that Assumption \ref{ass} is satisfied by infinitely many $p$.

Suppose now that $N$ can be written as $N=MD$ where $D\geq1$ is a square-free product of an \emph{even} number of primes and $(M,D)=1$. Let $K$ be an imaginary quadratic field, with ring of integers $\cO_K$, such that

\begin{assumption} \label{ass2} 
\begin{enumerate}
\item the primes dividing $pM$ split in $K$;
\item the primes dividing $D$ are inert in $K$.
\end{enumerate}
\end{assumption}

In particular, Assumption \ref{ass2} says that $K$ satisfies a \emph{modified Heegner hypothesis} relative to $N$. In many of the arguments below, one only uses the fact that $E$ has no $p$-torsion over $K_\infty$, which, by \cite[Lemma 2.1]{IP}, is true even without condition (2) in Assumption \ref{ass} provided that $p$ splits in $K$. However, in order to get better control on the field obtained by adding to the $m$-th layer $K_m$ of $K_\infty/K$ the coordinates of the $p^m$-torsion points of $E$ we will make use of this assumption. More generally, we expect that condition (2) in Assumption \ref{ass} can be somewhat relaxed, for example by just requiring that $\rho_{E,p}$ has non-solvable image (as done, e.g., in \cite{Cip} and \cite{CW}).

The last assumption we need to impose, which holds when $p$ does not divide the class number of $K$, is

\begin{assumption} \label{ass3}
The two primes of $K$ above $p$ are totally ramified in $K_\infty$.
\end{assumption}

This is a natural condition to require when working in the supersingular setting (cf., e.g., \cite[Assumptions 1.7, (2)]{DI}, \cite[Hypothesis (S)]{IP}, \cite[Theorem 1.2, (2)]{PW}); for example, it will allow us to apply the results of \cite{IP}. 

When $D=1$, the results obtained by \c{C}iperiani in \cite{Cip} tell us that 
\begin{itemize}
\item the $p$-primary Shafarevich--Tate group $\Sha_{p^\infty}(E/K_\infty)$ is a cotorsion $\Lambda$-module;
\item the $\Lambda$-coranks of the $p$-primary Selmer group $\Sel_{p^\infty}(E/K_\infty)$ and of $E(K_\infty)\otimes\Q_p/\Z_p$ are both $2$.
\end{itemize}
Under Assumptions \ref{ass}--\ref{ass3}, the present article offers an alternative approach to the study of anticyclotomic Selmer groups of elliptic curves at supersingular primes. More precisely, following Kobayashi (\cite{Kob03}) and Iovita--Pollack (\cite{IP}), we introduce restricted (plus/minus) Selmer groups $\Sel_{p^\infty}^\pm(E/K_\infty)$, whose Pontryagin duals $\mathcal X_\infty^\pm$ turn out to be finitely generated $\Lambda$-modules. 

Our main result, which corresponds to Theorem \ref{thm4.1}, is 

\begin{theorem} \label{main}
Each of the two $\Lambda$-modules $\mathcal X_\infty^\pm$ has rank $1$. 
\end{theorem}  

This can be viewed as the counterpart in the supersingular case of \cite[Theorem A]{Ber1}; as such, it provides yet another confirmation of the philosophy according to which Kobayashi's restricted Selmer groups are the ``right'' objects to consider in the non-ordinary setting.

Our strategy for proving Theorem \ref{main} is inspired by the work of Bertolini in \cite{Ber1} and goes as follows. First of all, we construct $\Lambda$-submodules $\E_\infty^\pm$ of the restricted Selmer groups $\Sel_{p^\infty}^\pm(E/K_\infty)$ out of suitable sequences of \emph{plus/minus Heegner points} on $E$. On the other hand, results of Cornut (\cite{Co}) and of Cornut--Vatsal (\cite{CV}) on the non-triviality of Heegner points as one ascends $K_\infty$ imply that the Pontryagin dual $\mathcal H_\infty^\pm$ of $\E_\infty^\pm$ has rank $1$ over $\Lambda$. Finally, a $\Lambda$-adic Euler system argument, to which the largest portion of our paper is devoted, allows us to prove that there is a natural surjective homomorphism of $\Lambda$-modules  
\[ \mathcal X_\infty^\pm\longepi\mathcal H_\infty^\pm \]
whose kernel turns out to be torsion, and Theorem \ref{main} follows.
 
It is worth remarking that the main difference between the ordinary and the supersingular settings is that, in the latter situation, Heegner points over $K_\infty$ are not naturally trace-compatible. In particular, there is no direct analogue of the $\Lambda$-module of Heegner points considered in \cite{Ber1} and \cite{PR}. In this paper we explain how to define subsequences of plus/minus Heegner points that satisfy a kind of trace-compatibility relation of the sort needed to study restricted Selmer groups as in \cite{Ber1}. 

As an application, combining our main theorem with results of \c{C}iperiani and of Iovita--Pollack, we obtain the following ``big O'' formula (Theorem \ref{corank-coro}) for the $\Z_p$-corank of the $p$-primary Selmer groups of $E$ over the finite layers of $K_\infty/K$. 

\begin{theorem} \label{corank-intro}
If $D=1$ then $\mathrm{corank}_{\Z_p}\bigl(\Sel_{p^\infty}(E/K_m)\bigr)=p^m+O(1)$. 
\end{theorem}

This is the supersingular analogue of a well-known result for ordinary primes; in fact, Theorem \ref{corank-intro} proves \cite[Conjecture 2.1]{Ber2} when $p$ is supersingular and $K_\infty$ is the anticyclotomic $\Z_p$-extension of $K$. Here we would like to emphasize that, due to the failure of Mazur's control theorem in its ``classical'' formulation, knowledge of the $\Lambda$-corank of $\Sel_{p^\infty}(E/K_\infty)$ as provided by \cite[Theorem 3.1]{Cip} is not sufficient to yield the growth result described in Theorem \ref{corank-intro} (see Remark \ref{corank-remark} for more details). Moreover, assuming the finiteness of the $p$-primary Shafarevich--Tate group of $E$ over $K_m$ for $m\gg0$, standard relations between Mordell--Weil, Selmer and Shafarevich--Tate groups of elliptic curves over number fields lead (at least when $D=1$) to a formula (Corollary \ref{rank-coro}) for the growth of the rank of $E(K_m)$.

As already mentioned, the techniques employed in this paper are close to those of Bertolini \cite{Ber1}. Similar results could presumably be obtained via different approaches, for example by adapting the arguments of \c{C}iperiani in \cite{Cip} (which rely on the techniques developed in \cite{CW}) or, following Mazur--Rubin (\cite{MR}), by using $\Lambda$-adic Kolyvagin systems as is done by Howard in \cite{Ho1}. In particular, we hope that extending the point of view of \cite{Ho1} to the supersingular setting would lead to an understanding of the torsion submodule of $\mathcal X_\infty^\pm$: we plan to come back to these issues in a future project.  

\subsubsection*{Acknowledgements}

It is a pleasure to thank Mirela \c{C}iperiani for helpful discussions and comments on some of the topics of this paper. We would also like to thank Christophe Cornut for useful correspondence on his joint work with Vinayak Vatsal.
 
\section{Anticyclotomic Iwasawa algebras} \label{sec2}

We briefly review the definition of the anticyclotomic $\Z_p$-extension $K_\infty$ of $K$ and then introduce the Iwasawa algebras that will be used in the rest of the paper.

\subsection{The anticyclotomic $\Z_p$-extension of $K$}

For every integer $m\geq 0$ let $H_{p^m}$ denote the ring class field of $K$ of conductor $p^m$, then set $H_{p^\infty}:=\cup_{m\geq0}H_{p^m}$. There is an isomorphism
\[ \Gal(H_{p^\infty}/K)\simeq\Z_p\times\Delta \]
where $\Delta$ is a finite group. 

The anticyclotomic $\Z_p$-extension $K_\infty/K$ is the unique $\Z_p$-extension of $K$ contained in $H_{p^\infty}$. We can write $K_\infty:=\cup_{m\geq0}K_m$, where $K_m$ is the unique subfield of $K_\infty$ such that 
\[ G_m:=\Gal(K_m/K)\simeq\Z/p^m\Z. \]
In particular, $K_0=K$. Set
\[ G_\infty:=\varprojlim_m G_m=\Gal(K_\infty/K)\simeq\Z_p. \] 
Finally, for every $m\geq0$ let $\Gamma_m:=\Gal(K_\infty/K_m)$, which is the kernel of the canonical projection $G_\infty\twoheadrightarrow G_m$.

\subsection{Iwasawa algebras and cyclotomic polynomials}

With notation as before, define $\Lambda_m:=\Z_p[G_m]$ and 
\[ \Lambda:=\varprojlim_m\Lambda_m=\Z_p[\![G_\infty]\!]. \]
Here the inverse limit is taken with respect to the maps induced by the natural projections
$G_{m+1}\rightarrow G_m$. For all $m\geq1$ fix a generator $\gamma_m$ of $G_n$ in such a way that ${\gamma_{m+1}|}_{K_m}=\gamma_m$; then $\gamma_\infty:=(\gamma_1,\dots,\gamma_m,\dots)$ is a topological generator of $G_\infty$. It is well known that the map
$\Lambda \rightarrow\Z_p[\![X]\!]$ defined by $\gamma_\infty\mapsto1+X$ is an isomorphism of $\Z_p$-algebras (see, e.g., \cite[Proposition 5.3.5]{NSW}). We will always identify these two $\Z_p$-algebras via this fixed isomorphism. 

Let $\Phi_m(X)=\sum_{i=0}^{p-1}X^{ip^{m-1}}$ be the $p^m$-th cyclotomic polynomial and set 
\[ \tilde\omega_m^+(X):=\prod_{\substack{2\leq n\leq m\\[1mm]\text{$n$ even}}}\Phi_n(1+X),\qquad\tilde\omega_m^-(X):=\prod_{\substack{1\leq n\leq m\\[1mm]\text{$n$ odd}}}\Phi_n(1+X),\qquad\omega^\pm_m(X):=X\cdot\tilde\omega_m^\pm(X), \]
\[ \omega_m(X):=\omega_m^\pm(X)\cdot\tilde\omega_m^\mp(X)=\omega_m^\mp(X)\cdot\tilde\omega_m^\pm(X)=X\cdot\prod_{1\leq n\leq m}\Phi_n(1+X)=(X+1)^{p^m}-1. \]
Then $\Lambda_m$ is isomorphic to $\Z_p[\![X]\!]/(\omega_m)$ under the isomorphism $\Lambda\simeq\Z_p[\![X]\!]$ described above. We also define 
\[ \Lambda_m^\pm:=\Z_p[\![X]\!]/(\omega_m^\pm). \] 
There are surjections $\Lambda\twoheadrightarrow\Lambda_m\twoheadrightarrow\Lambda_m^\pm$ and a canonical isomorphism $\Lambda_m^\pm\simeq\tilde \omega_m^\mp\Lambda_m$ given by multiplication by $\tilde\omega_m^\mp$. 

\begin{remark} \label{cong-remark}
If $m$ is even then $\tilde\omega_m^+=\tilde\omega_{m+1}^+$, hence $\omega_m^+=\omega_{m+1}^+$ and $\Lambda_m^+=\Lambda_{m+1}^+$. On the other hand, if $m$ is odd then $\tilde\omega_{m+1}^+\equiv p\tilde\omega_m^+$ in $\Lambda_m$, by which we mean that $\tilde\omega_{m+1}^+$ and $p\tilde\omega_m^+$ have the same image in $\Lambda_m$ (hence in $\Lambda_m^+$ and $\Lambda_m^-$ as well). Analogous relations (with the roles of ``even'' and ``odd'' reversed) hold in the case of sign $-$.
\end{remark}

For every integer $m\geq1$ set $D_m:=\Gal(K_m/\Q)$ and $\tilde\Lambda_m:=\Z_p[D_m]$, then define $D_\infty:=\Gal(K_\infty/\Q)=\varprojlim_m\Gal(K_m/\Q)$ and let $\tilde\Lambda:=\varprojlim_m\tilde\Lambda_m=\Z_p [\![D_\infty]\!]$ be the Iwasawa algebra of $D_\infty$ with coefficients in $\Z_p $. Recall that for every $m\geq1$ there is a canonical isomorphism
\[ \Gal(K_m/\Q)\simeq G_m\rtimes\Gal(K/\Q), \]
the natural action of $\Gal(K/\Q)=\langle\tau\rangle$ on $G_m$ by conjugation being equal to $\gamma^\tau=\gamma^{-1}$ for all $\gamma\in G_m$. Similarly, $\Gal(K_\infty/\Q)\simeq G_\infty\rtimes\Gal(K/\Q)$ with $\gamma^\tau=\gamma^{-1}$ for all $\gamma\in G_\infty$. 

With this in mind, write $\Lambda^{(\pm)}$ for the ring $\Lambda$ viewed as a module over $\tilde\Lambda$ via the action of $\Gal(K/\Q)$ given by $\gamma^\tau=\pm\gamma^{-1}$ for all $\gamma\in G_\infty$, so that $\Lambda^{(+)}$ corresponds to the linear extension of the natural action of $\Gal(K/\Q)$ on $G_\infty$ described above. Analogously, write $\Lambda_m^{(\pm)}$ for the $\tilde\Lambda_m$-module $\Lambda_m$ on which $\Gal(K/\Q)$ acts as $\gamma^\tau:=\pm\gamma^{-1}$ for all $\gamma\in G_m$. One also equips $\Lambda_m^\pm$ with a similar structure 
of $\tilde\Lambda_m$-module by defining as above $(\Lambda_m^\pm)^{(\epsilon)}$ to be the $\tilde \Lambda_m$-module $\Lambda_m^\pm$  with $\tau$ action by $\gamma^\tau:=\pm \gamma^{-1}$. 

We also consider the mod $p^m$ reductions of the above rings given by
\[ R_m:=\Lambda_m\otimes_\Z\Z/p^m\Z,\qquad\tilde R_m:=\tilde\Lambda_m\otimes_\Z\Z/p^m\Z,\qquad R_m^\pm:=\Lambda_m^\pm\otimes_\Z\Z/p^m\Z. \]
In particular, $\Lambda=\varprojlim_m R_m$. Similarly, we define 
\[ R_m^{(\epsilon)}:=\Lambda_m^{(\epsilon)}\otimes_\Z\Z/p^m\Z,\qquad(R_m^\pm)^{(\epsilon)}:=(\Lambda_m^\pm)^{(\epsilon)}\otimes_\Z\Z/p^m\Z,\qquad\tilde R_m^\pm:=R_m^\pm\otimes_\Lambda\tilde\Lambda. \]
Finally, for any compact or discrete $\Lambda$-module $M$ write $M^\vee:=\Hom_{\Z_p}^{\rm cont}(M,\Q_p/\Z_p)$ for its Pontryagin dual, equipped with the compact-open topology (here $\Hom_{\Z_p}^{\rm cont}$ denotes continuous homomorphisms of $\Z_p$-modules and $\Q_p/\Z_p$ is equipped with the quotient, i.e., discrete, topology). 

\section{Plus/Minus Selmer groups and control theorem} \label{sec3}

In this section we define the Selmer groups that we are interested in and state a control theorem for them. 

\subsection{Classical Selmer groups} 
 
For every integer $m\geq0$ let $\Sel_{p^\infty}(E/K_m)$ denote the $p$-primary Selmer group of $E$ over $K_m$ (see, e.g., \cite[Ch. 2]{Greenberg}). Moreover, let 
\begin{equation} \label{kappa-m-eq}
\kappa_m:E(K_m)\otimes\Q_p/\Z_p\longmono\Sel_{p^\infty}(E/K_m) 
\end{equation}
be the usual Kummer map and, for any prime $\lambda$ of $K_m$, with a slight abuse of notation write  
\begin{equation}\label{restriction}
\res_{m,\lambda}:\Sel_{p^\infty}(E/K_m)\longrightarrow E(K_{m,\lambda})\otimes\Q_p/\Z_p\end{equation}
for the composition of the restriction map with the inverse of the local Kummer map
\[ \kappa_{m,\lambda}:E(K_{m,\lambda})\otimes\Q_p/\Z_p\longmono H^1(K_{m,\lambda},E_{p^\infty}). \]
Similarly, for all $n\geq0$ there is a Kummer map
\begin{equation} \label{kappa-m-n-eq}
\kappa_{m,n}:E(K_m)/p^nE(K_m)\longmono\Sel_{p^n}(E/K_m),
\end{equation}
where $\Sel_{p^n}(E/K_m)$ is the $p^n$-Selmer group of $E$ over $K_m$. 

More generally, given a prime number $\ell$, we set $K_{m,\ell}:=K_m\otimes_\Q\Q_\ell=\prod_{\lambda|\ell}K_{m,\lambda}$ and let 
\[ \res_{m,\ell}=\oplus_{\lambda|\ell}\res_{m,\lambda}:H^1(K_m,E_{p^\infty})\longrightarrow H^1(K_{m,\ell},E_{p^\infty})=\oplus_{\lambda|\ell}H^1(K_{m,\lambda},E_{p^\infty}) \] 
be the direct sum of the local restrictions $\res_{m,\lambda}$ and 
\begin{equation} \label{restriction-2}
\res_{m,\ell}=\oplus_{\lambda\mid\ell}\res_{m,\lambda}:\Sel_{p^\infty}(E/K_m)\longrightarrow 
E(K_{m,\ell})\otimes\Q_p/\Z_p
\end{equation} 
be the direct sum of the maps in \eqref{restriction}, where 
\[ E(K_{m,\ell})\otimes\Q_p/\Z_p:=\bigoplus_{\lambda\mid\ell}E(K_{m,\lambda})\otimes\Q_p/\Z_p \] 
and $\lambda$ rages over the primes of $K_m$ above $\ell$. In the rest of the paper, we adopt a similar notation for other, closely related groups as well (e.g., with obvious definitions, we write $\res_{m,\ell}$ for the restriction map on $E(K_{m,\ell})/p^mE(K_{m,\ell})$ taking values in $H^1(K_{m,\ell},E_{p^m})$).

\begin{lemma} \label{no-torsion-lemma}
The group $E_{p^n}(K_m)$ is trivial for all $m,n\geq0$.
\end{lemma}

\begin{proof} Since, by part (1) of Assumption \ref{ass2}, the prime $p$ splits in $K$, this is \cite[Lemma 2.1]{IP}. Alternatively, one can use the surjectivity of $\rho_{E,p}$ ensured by part (2) of Assumption \ref{ass} and proceed as in the proof of \cite[Lemma 4.3]{Gross}. \end{proof}

In the next lemma we record some useful facts about Selmer groups.
 
\begin{lemma} \label{inj-selmer-lemma} \label{isom-selmer-lemma}
\begin{enumerate}
\item For all $m\geq0$ there is an injection
\[ \rho_m:\Sel_{p^m}(E/K_m)\longmono \Sel_{p^{m+1}}(E/K_{m+1}) \]
induced by the restriction map and the inclusion $E_{p^m}\subset E_{p^{m+1}}$.
\item For all $m\geq0$ restriction induces an injection 
\[\res_{K_{m+1}/K_m}:\Sel_{p^\infty}(E/K_{m})\longmono\Sel_{p^\infty}(E/K_{m+1}).\]
\item For all $m,n\geq0$ there is an isomorphism 
\[ \Sel_{p^n} (E/K_m)\simeq\Sel_{p^\infty}(E/K_m{)}_{p^n}. \]
\end{enumerate}
\end{lemma} 

\begin{proof} All three statements follow easily from Lemma \ref{no-torsion-lemma} (see, e.g., \cite[\S 2.3, Lemma 1]{Ber1}). \end{proof}

For all $m,n\geq0$ there is a commutative square
\begin{equation} \label{selmer-square-eq}
\xymatrix@C=33pt@R=30pt{E(K_m)/p^nE(K_m)\ar@{^(->}[d]\ar@{^(->}[r]^-{\kappa_{m,n}}&\Sel_{p^n}(E/K_m)\ar@{^(->}[d]\\E(K_m)\otimes\Q_p/\Z_p\ar@{^(->}[r]^-{\kappa_m}&\Sel_{p^\infty}(E/K_m)}
\end{equation}
in which the right vertical injection is induced by the isomorphism in part (3) of Lemma \ref{inj-selmer-lemma}.

Define the discrete $\Lambda$-module 
\[ \Sel_{p^\infty}(E/K_\infty):=\varinjlim_m\Sel_{p^\infty}(E/K_m), \]
the direct limit being taken with respect to the restriction maps in cohomology, which are injective by part (2) of Lemma \ref{inj-selmer-lemma}.

\subsection{Restricted Selmer groups} \label{restricted}

Let $p\mathcal O_K=\p\bar\p$ with $\p\not=\bar\p$; by Assumption \ref{ass3}, both $\p$ and $\bar\p$ are totally ramified in $K_\infty/K$. Write $K_\p$ and $K_{\bar\p}$ for the completions of $K$ at $\p$ and $\bar\p$, respectively. For all $m\geq0$ let $K_{m,\p}$ and $K_{m,\bar\p}$ be the completions of $K_m$ at the unique prime above $\p$ and $\bar\p$, respectively. To simplify notation, in the following lines we let $L$, $L_m$ denote one of these pairs of completions (i.e., $K_\p$, $K_{m,\p}$ or $K_{\bar\p}$, $K_{m,\bar\p}$); then $\Gal(L_m/L)\simeq\Z/p^m\Z$. 

For all integers $m,n$ with $m\geq n\geq0$ let $\tr_{L_m/L_n}:E(L_m)\rightarrow E(L_n)$ denote the trace map. Following Kobayashi (\cite{Kob03}), we define 
\begin{equation} \label{E-plus/minus-eq}
\begin{split}
E^+(L_m)&:=\bigl\{P\in E(L_m)\mid\text{$\tr_{L_m/L_n}(P)\in E(L_{n-1})$ for all odd $n$ with $1\leq n<m$}\bigr\},\\
E^-(L_m)&:=\bigl\{P\in E(L_m)\mid\text{$\tr_{L_m/L_n}(P)\in E(L_{n-1})$ for all even $n$ with $0\leq n<m$}\bigr\}.
\end{split}
\end{equation}

\begin{definition} \label{plus/minus-dfn}
The \emph{plus/minus $p$-primary Selmer groups} of $E$ over $K_m$ are
\[ \Sel_{p^\infty}^\pm(E/K_m):=\ker\left(\Sel_{p^\infty}(E/K_m)\xrightarrow{\res_{m,p}}\frac{E(K_{m,\p})\otimes\Q_p/\Z_p}{E^\pm(K_{m,\p})\otimes\Q_p/\Z_p}\bigoplus\frac{E(K_{m,\bar\p})\otimes\Q_p/\Z_p}{E^\pm(K_{m,\bar\p})\otimes\Q_p/\Z_p}\right), \]
where $\res_{m,p}$ is the composition of the restrictions at $\p$ and $\bar\p$ with the quotient projections.
\end{definition}

In an analogous manner, replacing $\Sel_{p^\infty}(E/K_m)$ with $\Sel_{p^n}(E/K_m)$ and $\Q_p/\Z_p$ with $\Z/p^n\Z$, one can define $\Sel_{p^n}^\pm(E/K_m)$ for all $n\geq1$.

\begin{remark} \label{equivalence-rem}
In \cite{IP}, the groups $\Sel_{p^\infty}^\pm(E/K_m)$ are defined in terms of the formal group $\hat E$ of $E$. More precisely, if $\mathfrak m$ and $\bar{\mathfrak m}$ denote the maximal ideals of the rings of integers of $K_{m,\p}$ and $K_{m,\bar\p}$, respectively, Iovita and Pollack introduce subgroups $\hat E^\pm(\mathfrak m)\subset\hat E(\mathfrak m)$ and $\hat E^\pm(\bar{\mathfrak m})\subset\hat E(\bar{\mathfrak m})$ as in \eqref{E-plus/minus-eq}. Then they use these subgroups to define $\Sel_{p^\infty}^\pm(E/K_m)$ as in Definition \ref{plus/minus-dfn}, replacing $E(K_{m,\p})$ (respectively, $E^\pm(K_{m,\p})$) with $\hat E(\mathfrak m)$ (respectively, $\hat E^\pm(\mathfrak m)$) and $E(K_{m,\bar\p})$ (respectively, $E^\pm(K_{m,\bar\p})$) with $\hat E(\bar{\mathfrak m})$ (respectively, $\hat E^\pm(\bar{\mathfrak m})$). To see that their definition is equivalent to Definition \ref{plus/minus-dfn}, recall that, by \cite[Ch. VII, Proposition 2.2]{Sil}, the group $\hat E(\mathfrak m)$ is isomorphic to the kernel $E_1(K_{m,\p})$ of the reduction map $E(K_{m,\p})\rightarrow\bar E(\mathbb F_p)$. On the other hand, $|\bar E(\mathbb F_p)|=p+1$ because $a_p=0$, and $\hat E^\pm(\mathfrak m)\simeq E_1(K_{m,\p})\cap E^\pm(K_{m,\p})$, hence there are isomorphisms 
\[ \hat E(\mathfrak m)\otimes\Q_p/\Z_p\overset\simeq\longrightarrow E(K_{m,\p})\otimes\Q_p/\Z_p,\quad\hat E^\pm(\mathfrak m)\otimes\Q_p/\Z_p\overset\simeq\longrightarrow E^\pm(K_{m,\p})\otimes\Q_p/\Z_p. \]
Analogous considerations apply to $\hat E(\bar{\mathfrak m})$ and $\hat E^\pm(\bar{\mathfrak m})$, and the desired equivalence follows.
\end{remark}

Now form the two discrete $\Lambda$-modules 
\[ \Sel_{p^\infty}^\pm(E/K_\infty):=\varinjlim_m \Sel_{p^\infty}^\pm(E/K_m)\subset\Sel_{p^\infty}(E/K_\infty), \]
the direct limits being taken with respect to the restriction maps in cohomology. Note that, thanks to part (2) of Lemma \ref{inj-selmer-lemma}, these restrictions are injective. Furthermore, the fact that the groups $\Sel_{p^\infty}^\pm(E/K_m)$ do indeed form a direct system follows directly from Definition \ref{plus/minus-dfn} and the compatibility properties of the restriction maps involved.

\subsection{Control theorem} \label{control-subsec}

The next result provides a substitute for Mazur's original ``control theorem'' (\cite{Maz1}) and extends \cite[Theorem 9.3]{Kob03} to our anticyclotomic setting. 

\begin{theorem}[Iovita--Pollack] \label{CT-thm}
For every integer $m\geq0$ the restriction
\begin{equation} \label{CT}
\res_{K_\infty/K_m}:\Sel_{p^\infty}^\pm(E/K_m)^{\omega_m^\pm=0}\longrightarrow\Sel^\pm_{p^\infty}(E/K_\infty)^{\omega_m^\pm=0}
\end{equation}
is injective and has finite cokernel bounded independently of $m$.
\end{theorem}

\begin{proof} Keeping Remark \ref{equivalence-rem} in mind, this is \cite[Theorem 6.8]{IP}. \end{proof}

Note that, by definition, $\Sel_{p^\infty}^\pm(E/K_m)^{\omega_m^\pm=0}$ is a $\Lambda_m^\pm$-module. Consider the Pontryagin dual
\[ \mathcal X_\infty^\pm:=\Hom_{\Z_p}^{\rm cont}\bigl(\Sel_{p^\infty}^\pm(E/K_\infty),\Q_p/\Z_p\bigr) \]
of $\Sel_{p^\infty}^\pm(E/K_\infty)$, equipped with its canonical structure of compact $\Lambda$-module. Moreover, for every integer $m\geq0$ write
\[ \mathcal X_m^\pm:=\Hom_{\Z_p}^{\rm cont}\bigl(\Sel_{p^\infty}^\pm(E/K_m),\Q_p/\Z_p\bigr) \]
for the Pontryagin dual of $\Sel_{p^\infty}^\pm(E/K_m)$, so that $\mathcal X_m^\pm$ has a natural $\Lambda_m$-module structure. By duality, the map in \eqref{CT} gives a surjection   
\begin{equation} \label{coro-CT}
\res_{K_\infty/K_m}^\vee:\mathcal X_\infty^\pm\big/\omega^\pm_m\mathcal X_\infty^\pm\longepi\mathcal X_m^\pm\big/\omega_m^\pm\mathcal X_m^\pm
\end{equation} 
whose finite kernel can be bounded independently of $m$. 

\begin{proposition}
The $\Lambda$-module $\mathcal X_\infty^\pm$ is finitely generated.
\end{proposition}

\begin{proof} Since $\omega_m^\pm$ is topologically nilpotent and $\mathcal X_m^\pm$ is finitely generated as a $\Z_p$-module, the claim follows from \eqref{coro-CT} and \cite[Corollary, p. 226]{BH}. \end{proof}

There are canonical commutative squares
\[ \xymatrix@C=32pt@R=30pt{\Sel_{p^\infty}^\pm(E/K_m)^{\omega_m^\pm=0}\ar@{^(->}[rr]^-{\res_{K_{\infty}/K_m} }\ar@{^(->}[d]^-{\res_{K_{m+1}/K_m}}&& 
\Sel_{p^\infty}(E/K_\infty)^{\omega_m^\pm=0}\ar@{^(->}[d]^-{i_m}\\
\Sel_{p^\infty}^\pm(E/K_{m+1})^{\omega_{m+1}^\pm=0}\ar@{^(->}[rr]^-{\res_{K_{\infty}/K_{m+1}} }&&
\Sel_{p^\infty}(E/K_\infty)^{\omega_{m+1}^\pm=0}} \] 
where $i_m$ is the natural inclusion (here observe that $\omega_m^\pm\,|\,\omega_{m+1}^\pm$) and 
\[ \xymatrix@C=35pt@R=30pt{\mathcal X_\infty^\pm\big/\omega_{m+1}^\pm\mathcal X_\infty^\pm\ar@{->>}[rr]^-{\res^\vee_{K_\infty/K_{m+1}}}\ar@{->>}[d]^-{i_m^\vee}&&
\mathcal X_{m+1}^\pm\big/\omega_{m+1}^\pm\mathcal X_{m+1}^\pm\ar[d]^-{\res^\vee_{K_{m+1}/K_m}}\ar@{->>}[d]^-{\res^\vee_{K_{m+1}/K_m}}\\
\mathcal X_\infty^\pm\big/\omega_{m}^\pm\mathcal X_\infty^\pm\ar@{->>}[rr]^-{\res^\vee_{K_\infty/K_{m}}}&&
\mathcal X_{m}^\pm\big/\omega_{m}^\pm\mathcal X_{m}^\pm} \] 
where, as before, the symbol $\phi^\vee$ denotes the Pontryagin dual of a given map $\phi$. 

\section{Iwasawa modules of plus/minus Heegner points} \label{sec4}

In order to relax, as in Assumption \ref{ass2}, the Heegner hypothesis imposed in \cite{Ber1} and \cite{Cip}, we need to consider Heegner points on Shimura curves attached to division quaternion algebras over $\Q$. These points will then be mapped to the elliptic curve $E$ via a suitable modular parametrization.

\subsection{Shimura curves and modularity} \label{shimura-subsec}

Let $B$ denote the (indefinite) quaternion algebra over $\Q$ of discriminant $D$ and fix an isomorphism of $\R$-algebras
\[ i_\infty:B\otimes_\Q\R\overset\simeq\longrightarrow\M_2(\R). \]
Let $R(M)$ be an Eichler order of $B$ of level $M$ and write $\Gamma_0^D(M)$ for the group of norm $1$ elements of $R(M)$. If $D>1$ and $\HH:=\{z\in\C\mid\mathrm{Im}(z)>0\}$ then the Shimura curve of level $M$ and discriminant $D$ is the (compact) Riemann surface
\[ X_0^D(M):=\Gamma_0^D(M)\backslash\mathcal H. \]
Here the action of $\Gamma_0^D(M)$ on $\mathcal H$ by M\"obius (i.e., fractional linear) transformations is induced by $i_\infty$. If $D=1$ (i.e., $M=N$) then we can take $B=\M_2(\Q)$, so that $\Gamma_0^1(M)=\Gamma_0(N)$ and $\Gamma_0^1(M)\backslash\mathcal H=Y_0(N)$, the open modular curve of level $N$; in this case, we define $X_0^1(M):=X_0(N)$, the (Baily--Borel) compactification of $Y_0(N)$ obtained by adding its cusps. By a result of Shimura, the projective algebraic curve corresponding to $X^D_0(M)$ is defined over $\Q$.

Finally, thanks to the modularity of $E$, Faltings's isogeny theorem and (when $D>1$) the Jacquet--Langlands correspondence between classical and quaternionic modular forms, there exists a surjective morphism
\begin{equation} \label{parametrization} 
\pi_E:X_0^D(M)\longrightarrow E 
\end{equation}
defined over $\Q$, which we fix once and for all (see, e.g., \cite[\S 4.3]{LRV} and \cite[\S 3.4.4]{Zhang1} for details).

\subsection{Heegner points and trace relations} \label{heegner-trace-subsec}

Let us first consider the case where $D=1$. Choose an ideal $\mathcal N\subset\cO_K$ such that $\mathcal O_K/\mathcal N\simeq \Z/N\Z$, which exists thanks to the Heegner hypothesis satisfied by $K$. For each integer $c\geq 1$ prime to $N$ and the discriminant of $K$, let $\cO_{c}=\Z+c\cO_K$ be the order of $K$ of conductor $c$. The isogeny $\C/\cO_{c}\rightarrow \C/(\cO_{c}\cap\mathcal N)^{-1}$ defines a Heegner point $x_c\in Y_0(N)\subset X_0(N)$ that, by complex multiplication, is rational over the ring class field $H_{c}$ of $K$ of conductor $c$. In the rest of the paper, $c$ will vary in the powers of the prime $p$. 

In the general quaternionic case, a convenient way to introduce Heegner points $x_c\in X_0^D(M)(H_c)$ is to exploit the theory of (oriented) optimal embeddings of quadratic orders into Eichler orders. We shall not give precise definitions here, but rather refer to \cite[Section 2]{BD96} for details. From now on we fix a compatible system of Heegner points
\[ {\bigl\{x_{p^m}\in X_0^D(M)(H_{p^m})\bigr\}}_{m\geq0} \]
as described in \cite[\S 2.4]{BD96}.

Recall the morphism $\pi_E$ introduced in \eqref{parametrization} and for every integer $m\geq0$ set
\[ y_{p^m}:=\pi_E(x_{p^m})\in E(H_{p^m}). \]
In order to define Heegner points over $K_\infty$, we take Galois traces. Namely, for all $m\geq1$ set
\begin{equation} \label{d(m)-eq}
d(m):=\min\bigl\{d\in\mathbb N\mid K_m\subset H_{p^d}\bigr\}. 
\end{equation}
For example, if $p\nmid h_K$ then $d(m)=m+1$. In light of this, for all $m\geq0$ define
\[ z_m:=\tr_{H_{p^{d(m)}}/K_m}\bigl(y_{p^{d(m)}}\bigr)\in E(K_m). \]
By the formulas in \cite[\S 3.1, Proposition 1]{PR} and \cite[\S 2.5]{BD96}, the following relations hold:
\begin{equation} \label{heegner-eq}
\tr_{K_{m}/K_{m-1}}(z_{m})=\begin{cases}-z_{m-2}&\text{if $m\geq2$},\\[2mm]\displaystyle{\frac{p-1}{2}z_0}&\text{if $m=1$}. \end{cases} 
\end{equation}

\subsection{Plus/minus Heegner points and trace relations} \label{plus-minus-subsec}

Starting from the Heegner points that we considered in \S \ref{heegner-trace-subsec}, we define \emph{plus/minus Heegner points} $z_m^\pm$ as follows. Set $z_0^\pm:=z_0$ and for every $m\geq1$ define   
\[ z^+_m:=\begin{cases}z_m&\text{if $m$ is even},\\[2mm]z_{m-1}&\text{if $m$ is odd},\end{cases}\qquad z^-_m:=\begin{cases}z_{m-1}&\text{if $m$ is even},\\[2mm]z_{m}&\text{if $m$ is odd}.\end{cases} \] 
As a consequence of formulas \eqref{heegner-eq}, the points $z_m^\pm\in E(K_m)$ satisfy the following relations: 
\begin{itemize}
\item[(a)] $\tr_{K_m/K_{m-1}}(z_m^+)=-z_{m-1}^+$\quad for every even $m\geq2$;
\item[(b)] $\tr_{K_m/K_{m-1}}(z_m^+)=pz_{m-1}^+$\quad for every odd $m\geq1$;
\item[(c)] $\tr_{K_m/K_{m-1}}(z_m^-)=pz_{m-1}^-$\quad for every even $m\geq2$;
\item[(d)] $\tr_{K_m/K_{m-1}}(z_m^-)=-z_{m-1}^-$\quad for every odd $m\geq3$;
\item[(e)] $\tr_{K_1/K_0}(z_1^-)=\frac{p-1}{2}z_0^-=\frac{p-1}{2}z_0$. 
\end{itemize}
Finally, with $\kappa_{m,m}$ as in \eqref{kappa-m-n-eq} and $\kappa_m$ as in \eqref{kappa-m-eq}, for all $m\geq0$ set 
\[ \alpha_m^\pm:=\kappa_{m,m}\big([z_m^\pm]\big)\in\Sel_{p^m}(E/K_m),\qquad\beta_m^\pm:=\kappa_m\big(z_m^\pm\otimes1\big)\in\Sel_{p^\infty}(E/K_m). \]
For every $m\geq0$ let 
\[ \tilde\rho_m:\Sel_{p^{m+1}}(E/K_{m+1})\longrightarrow\Sel_{p^m}(E/K_m) \]
be the composition of $\cores_{K_{m+1}/K_m}$ with the multiplication-by-$p$ map. Moreover, write
\[ t_m:E(K_{m+1})/p^{m+1}E(K_{m+1})\longrightarrow E(K_m)/p^mE(K_m) \]
for the natural map induced by $\tr_{K_{m+1}/K_m}$. The resulting square
\begin{equation} \label{squaresquare-eq}
\xymatrix@C=60pt@R=35pt{E(K_{m+1})/p^{m+1}E(K_{m+1})\ar[d]^-{t_m}\ar@{^(->}[r]^-{\kappa_{m+1,m+1}}&\Sel_{p^{m+1}}(E/K_{m+1})\ar@{^(->}[d]^-{\tilde\rho_m}\\E(K_m)/p^mE(K_m)\ar@{^(->}[r]^-{\kappa_{m,m}}&\Sel_{p^m}(E/K_m)}
\end{equation}
is commutative.

The following result collects the properties enjoyed by the classes $\alpha_m^\pm$ under corestriction.
\begin{proposition} \label{cores-prop}
The following formulas hold:
\begin{itemize}
\item[(a)] $\tilde\rho_{m-1}(\alpha_m^+)=-\alpha_{m-1}^+$\quad for every even $m\geq2$;
\item[(b)] $\tilde\rho_{m-1}(\alpha_m^+)=p\alpha_{m-1}^+$\quad for every odd $m\geq1$;
\item[(c)] $\tilde\rho_{m-1}(\alpha_m^-)=p\alpha_{m-1}^-$\quad for every even $m\geq2$;
\item[(d)] $\tilde\rho_{m-1}(\alpha_m^-)=-\alpha_{m-1}^-$\quad for every odd $m\geq3$;
\item[(e)] $\tilde\rho_{m-1}(\alpha_1^-)=\frac{p-1}{2}\alpha_0^-=\frac{p-1}{2}\alpha_0$. 
\end{itemize}
\end{proposition}

Of course, analogous formulas, with $\cores_{K_m/K_{m-1}}$ in place of $\tilde\rho_{m-1}$, hold for $\beta_m^\pm$.

\begin{proof} Straightforward from the corresponding formulas for the points $z_m^\pm$ listed above and square \eqref{squaresquare-eq}. \end{proof}

Now we can prove

\begin{proposition} \label{class-selmer-prop}
\begin{enumerate}
\item The class $\alpha_m^\pm$ belongs to $\Sel^\pm_{p^m}(E/K_m)^{\omega_m^\pm=0}$ for every $m\geq0$.
\item The class $\beta_m^\pm$ belongs to $\Sel^\pm_{p^\infty}(E/K_m)^{\omega_m^\pm=0}$ for every $m\geq0$.
\end{enumerate}
\end{proposition} 

\begin{proof} Fix an integer $m\geq0$, let $\lambda$ denote either $\p$ or $\bar\p$, set $L:=K_{m,\lambda}$ and put $\alpha^\pm_\lambda:=\res_{m,\lambda}(\alpha^\pm_m)\in E(L)/p^mE(L)$. The previous formulas show that $[z_m^\pm]\in E^\pm(L)/p^mE^\pm(L)$, hence $\alpha_m^\pm\in\Sel^\pm_{p^m}(E/K_m)$. On the other hand, the fact that $\omega_m^\pm\alpha_m^\pm=0$ follows from a global version of the local computations in the proof of \cite[Proposition 4.11]{IP} (which is possible because the points $z_m^\pm$, as well as the trace relations they satisfy, are global). This proves (1), and (2) can be shown in the same way.  \end{proof}

\subsection{Direct limits of plus/minus Heegner modules} \label{direct-subsec}
 
In light of Proposition \ref{class-selmer-prop}, for all $m\geq0$ let $\E_m^\pm:=R_m^\pm\alpha^\pm_m$ denote the $R_m^\pm$-submodule (or, equivalently, the $R_m$-submodule) of $\Sel_{p^m}^\pm(E/K_m)^{\omega_m^\pm=0}$ generated by $\alpha_m^\pm$. The inclusion $\Sel_{p^m}^\pm(E/K_m)^{\omega_m^\pm=0}\subset\Sel_{p^m}^\pm(E/K_m)$ allows us to regard $\E_m^\pm$ as a submodule of the whole restricted Selmer group $\Sel_{p^m}^\pm(E/K_m)$. Note that, by the commutativity of \eqref{selmer-square-eq}, the injection $\Sel_{p^m}^\pm(E/K_m)\hookrightarrow\Sel_{p^\infty}^\pm(E/K_m)$ given by part (3) of Lemma \ref{inj-selmer-lemma} sends $\E_m^\pm$ to the $\Lambda_m^\pm$-submodule $\Lambda_m^\pm\beta^\pm_m$ generated by $\beta_m^\pm$.

For the proof of the next result, recall the injection
\[ \rho_m:\Sel_{p^m}(E/K_m)\longmono \Sel_{p^{m+1}}(E/K_{m+1}) \]
of part (1) of Lemma \ref{inj-selmer-lemma}. If $F'/F$ is a Galois extension of number fields and $M$ is a continuous $G_F$-module then $\res_{F'/F}\circ\cores_{F'/F}=\tr_{F'/F}$, where 
\[ \tr_{F'/F}:=\sum_{\sigma\in\Gal(F'/F)}\sigma:H^1(F',M)\rightarrow H^1(F',M) \] 
is the Galois trace map (see, e.g., \cite[Corollary 1.5.7]{NSW}). We immediately obtain
\begin{equation} \label{res-cores-norm} 
\rho_m\circ\tilde\rho_m=p\tr_{K_{m+1}/K_m}
\end{equation} 
for all $m\geq0$. For all $m\geq0$ and $n\geq1$ consider the canonical map 
\[ \iota_m^n:E(K_m)/p^nE(K_m)\longrightarrow E(K_m)/p^{n+1}E(K_m),\qquad [Q]\longmapsto[pQ]  \] 
and, finally, denote by 
\[ j_m:E(K_m)/p^mE(K_m)\longrightarrow E(K_{m+1})/p^mE(K_{m+1}) \]
the obvious map.

\begin{proposition} \label{inj-prop}
The map $\rho_m$ induces injections
\[ \rho_m^\pm:\E_m^\pm\longmono\E_{m+1}^\pm \]
for all $m\geq0$.
\end{proposition}

\begin{proof} Fix an $m\geq0$. We treat only the case of sign $+$, the other being analogous. By part (2) of Lemma \ref{inj-selmer-lemma}, $\rho_m$ is injective at the level of Selmer groups, so it suffices to show that $\rho_m(\E_m^+)\subset\E_{m+1}^+$. Suppose that $m$ is even. Then $z_m^+=z_{m+1}^+$, and the commutativity of the square
\[ \xymatrix@C=60pt@R=35pt{E(K_m)/p^mE(K_m)\ar@{^(->}[d]^-{\iota^m_{m+1}\circ j_m}\ar@{^(->}[r]^-{\kappa_{m,m}}&\Sel_{p^m}(E/K_m)\ar@{^(->}[d]^-{\rho_m}\\E(K_{m+1})/p^{m+1}E(K_{m+1})\ar@{^(->}[r]^-{\kappa_{m+1,m+1}}&\Sel_{p^{m+1}}(E/K_{m+1})} \]
implies that $\rho_m(\alpha_m^+)=p\alpha_{m+1}^+$. By definition of the Galois action on our cohomology groups, it follows that $\rho_m(\E_m^+)\subset\E_{m+1}^+$. Now suppose that $m$ is odd. By part (a) of Proposition \ref{cores-prop}, $\tilde\rho_m(\alpha_{m+1}^+)=-\alpha_m^+$. Applying \eqref{res-cores-norm} and using the fact that, by Proposition \ref{class-selmer-prop}, the action of $R_{m+1}$ on $\alpha_{m+1}^+$ factors through $R_{m+1}^+$, we get $\rho_m(\alpha_m^+)=-p\tr_{K_{m+1}/K_m}(\alpha_{m+1}^+)\in R_{m+1}\alpha_{m+1}^+=R_{m+1}^+\alpha_{m+1}^+$. As before, we conclude that $\rho_m(\E_m^+)\subset\E_{m+1}^+$. \end{proof}

Thanks to Proposition \ref{inj-prop}, we can form the discrete $\Lambda$-module 
\[ \mathcal E_\infty^\pm:=\dirlim_m\mathcal E_m^\pm, \]
where the direct limits are taken with respect to the maps $\rho_m^\pm$ of Proposition \ref{inj-prop}. Moreover, the commutativity of the squares
\[ \xymatrix@C=45pt@R=35pt{\Sel_{p^m}^\pm(E/K_m)\ar@{^(->}[d]^-{\rho^\pm_m}\ar@{^(->}[r]&\Sel_{p^\infty}^\pm(E/K_m)\ar@{^(->}[d]^-{\res_{K_{m+1}/K_m}}\\\Sel^\pm_{p^{m+1}}(E/K_{m+1})\ar@{^(->}[r]&\Sel^\pm_{p^\infty}(E/K_{m+1}),} \]
in which the horizontal injections are induced by the isomorphisms in part (3) of Lemma \ref{inj-selmer-lemma}, shows that $\E_\infty^\pm$ can be naturally viewed as a $\Lambda$-submodule of $\Sel_{p^\infty}^\pm(E/K_\infty)$.

Denote by 
\[ \mathcal H_\infty^\pm:=(\E_\infty^\pm)^\vee=\varprojlim_m(\E_m^\pm)^\vee \]
the Pontryagin dual of $\E_\infty^\pm$. We shall see below (Proposition \ref{H-rank-prop}) that both $\mathcal H_\infty^+$ and $\mathcal H_\infty^-$ are finitely generated, torsion-free $\Lambda$-modules of rank $1$. 

\subsection{Nontriviality of Heegner points and the $\Lambda$-rank of $\mathcal H_\infty^\pm$} \label{rank-1-subsec}

We want to apply the results of Cornut (\cite{Co}) and of Cornut--Vatsal (\cite{CV}) on the nontriviality of Heegner points on $E$ as one ascends $K_\infty$ to show that $\mathcal H_\infty^\pm$ have rank $1$ over $\Lambda$. Similar ideas can also be found in \cite[Proposition 2.1]{Cip} and \cite[Theorem 2.5.1]{CW}.

We begin with some lemmas.

\begin{lemma} \label{non-div-lemma}
For $m\gg0$ the point $z_m^\pm$ is not $p^m$-divisible in $E(K_m)$.
\end{lemma}

\begin{proof} Results of Cornut (\cite{Co}) and of Cornut--Vatsal (\cite{CV}) guarantee that the points $z_m^\pm\in E(K_m)$ are non-torsion for $m\gg0$. We prove the lemma for sign $+$, the case of sign $-$ being completely analogous. To fix ideas, define
\[ m_0:=\min\big\{m\in\N\mid\text{$m$ is even and $z^+_m$ is non-torsion}\big\}. \]
We claim that $z_m^\pm$ is not $p^m$-divisible in $E(K_m)$ for even $m\gg m_0$. First of all, the formulas in \S \ref{plus-minus-subsec} imply that if $n\in\N$ then
\[ \tr_{K_{m_0+2n}/K_{m_0}}\big(z_{m_0+2n}^+\big)=(-1)^np^nz_{m_0}^+. \] 
If $z_{m_0+2n}^+=p^{m_0+2n}x$ with $x\in E(K_{m_0+2n})$ and we set $y:=(-1)^n\tr_{K_{m_0+2n}/K_{m_0}}(x)\in E(K_{m_0})$ then $p^nz_{m_0}^+=p^{m_0+2n}y$, that is, $p^n(z_{m_0}^+-p^{m_0+n}y)=0$. On the other hand, by Lemma \ref{no-torsion-lemma}, the torsion group $E_{p^n}(K_{m_0})$ is trivial, so we conclude that $z_{m_0}^+$ is $p^{m_0+n}$-divisible in $E(K_{m_0})$. But the Mordell--Weil group $E(K_{m_0})$ is finitely generated and $z_{m_0}^+$ is non-torsion, hence $z_{m_0}^+$ is $p^t$-divisible in $E(K_{m_0})$ only for finitely many $t\in\N$. The lemma follows. \end{proof}

\begin{lemma} \label{E-m-nontrivial-lemma}
The $R_m^\pm$-module $\E_m^\pm$ is non-trivial for $m\gg0$.
\end{lemma}

\begin{proof} By Lemma \ref{non-div-lemma}, the class $[z_m^\pm]$ of $z_m^\pm$ in $E(K_m)/p^mE(K_m)$ is non-zero for $m\gg0$. Finally, the injectivity of the maps $\kappa_{m,m}$ implies that $\alpha_m^\pm=\kappa_{m,m}\big([z_m^\pm]\big)$ is non-zero in $\Sel^\pm_{p^m}(E/K_m)$. In particular, $\E_m^\pm$ is non-trivial for $m\gg0$.  \end{proof}

As an immediate consequence, we get

\begin{lemma} \label{E-nontrivial-lemma}
The $\Lambda$-modules $\E_\infty^\pm$ are non-trivial.
\end{lemma}

\begin{proof} By Proposition \ref{inj-prop}, the maps $\rho_m^\pm$ with respect to which the direct limits $\E_\infty^\pm$ are taken are injective, hence $\E^\pm_m$ injects into $\E_\infty^\pm$ for all $m\geq0$. The lemma follows from Lemma \ref{E-m-nontrivial-lemma}. \end{proof}

Now we can prove

\begin{proposition} \label{H-rank-prop}
The $\Lambda$-modules $\mathcal H_\infty^\pm$ are finitely generated, torsion-free and of rank $1$.
\end{proposition}

\begin{proof}  For every $m\geq0$ the natural surjections $R_m\twoheadrightarrow\E^\pm_m$ induce, by duality, injections $(\E_m^\pm)^\vee\hookrightarrow R^\vee_m$. Since $R_m=\Z/p^m\Z[G_m]$, there are isomorphisms $R_m^\vee\simeq R_m$, hence taking inverse limits gives injections $\mathcal H_\infty^\pm=\varprojlim_m(\E_m^\pm)^\vee\hookrightarrow\varprojlim_m R_m=\Lambda$. Since $\Lambda$ is a noetherian domain, this shows that $\mathcal H_\infty^\pm$ are finitely generated, torsion-free $\Lambda$-modules of rank equal to $0$ or to $1$. Now the structure theorem for finitely generated $\Lambda$-modules implies that if $\mathrm{rank}_\Lambda(\mathcal H_\infty^\pm)=0$ then $\mathcal H_\infty^\pm=0$, hence $\E_\infty^\pm=(\mathcal H_\infty^\pm)^\vee=0$. This contradicts Lemma \ref{E-nontrivial-lemma}, so we conclude that $\mathrm{rank}_\Lambda(\mathcal H_\infty^\pm)=1$. \end{proof}

\section{$\Lambda$-adic Euler systems} \label{sec5}

The goal of this section is to prove Theorem \ref{main}; we restate it below.

\begin{theorem} \label{thm4.1}
Each of the two $\Lambda$-modules $\mathcal X_\infty^\pm$ has rank $1$. 
\end{theorem} 

In other words, we show that 
\[ \mathrm{corank}_{\Lambda}\bigl(\Sel^+_{p^\infty}(E/K_\infty)\bigr)=\mathrm{corank}_{\Lambda}\bigl(\Sel^-_{p^\infty}(E/K_\infty)\bigr)=1. \]
The injection of $\Lambda$-modules $\E_\infty^\pm\hookrightarrow\Sel^\pm_{p^\infty}(E/K_\infty)$ gives, by duality, a surjection of $\Lambda$-modules 
\[ \pi^\pm:\mathcal X_\infty^\pm\longepi\mathcal H_\infty^\pm. \]
Proving Theorem \ref{thm4.1} is thus equivalent to showing that the $\Lambda$-module $\ker(\pi^\pm)$ is torsion. Equivalently, if $\tau$ denotes the generator of $\Gal(K/\Q)$ then we need to show that all elements of $\ker(\pi^\pm)$ lying in an eigenspace for $\tau$ are $\Lambda$-torsion. 

Choose an element $x\in \mathcal X_\infty^\pm$ such that $\tau x=\epsilon x$ for some $\epsilon\in\{\pm\}$ and $x$ is not $\Lambda$-torsion: this can be done because the $\Lambda$-module $\mathcal H_\infty^\pm$ has rank $1$ by Proposition \ref{H-rank-prop} and the map $\pi^\pm$ is surjective. As is explained in \cite[p. 170]{Ber1}, to prove Theorem \ref{thm4.1} it is enough to show that every $y\in\ker(\pi^\pm)^{-\epsilon}$ is $\Lambda$-torsion. To do this, in the next subsections we will adapt the $\Lambda$-adic Euler system argument of \cite{Ber1}. 

%Before beginning the proof, having fixed $\delta$ once and for all in the statement of Theorem \ref{thm4.1}, we make the following 
%conventions to simplify the notation:
%\[\mathcal X^\pm_\infty:=\mathcal X_\infty^\delta;\quad\mathcal X^\pm_m:=\mathcal X_m^\delta;\quad \mathcal E^\pm _\infty:=\mathcal E_\infty^\delta; \quad 
%\mathscr T_\infty:=\mathcal T_\infty^\delta;\]
%\[\Sel_m:=\Sel_{p^m}^\delta(E/K_m);\quad \Sel_\infty:=\Sel_{p^\infty}^\delta(E/K_\infty);\]
%\[\omega_m=\omega_m^\delta;\quad \pi:=\pi^\delta.\]
%\begin{remark}We advise the reader that the element denoted $\omega_m$  in \cite{IP} has an other meaning there.\end{remark}

\subsection{Kolyvagin primes} 

Denote by 
\[ \rho_m:G_\Q\longrightarrow \Aut(E_{p^m})\simeq\GL_2(\Z/p^m\Z) \] 
the Galois representation on $E_{p^m}$ and let $K(E_{p^m})$ be the composite of $K$ and the field cut out by $\rho_m$; in other words, $K(E_{p^m})$ is the composite of $K$ and $\bar\Q^{\ker(\rho_m)}$. In particular, $K(E_{p^m})$ is Galois over $\Q$. 

\begin{definition} \label{kolyvagin-dfn}
A prime number $\ell$ is a \emph{Kolyvagin prime} for $p^m$ if $\ell\nmid Np$ and $\Frob_\ell=[\tau]$ in $\Gal(K(E_{p^m})/\Q)$. 
\end{definition} 

In particular, Kolyvagin primes are inert in $K$ and hence split completely in $K_m$ for all $m\geq1$. Let $\ell$ be a Kolyvagin prime for $p^m$. Define the $\tilde R_m$-modules 
\[ \bigl(E(K_{m,\ell},E)/p^mE(K_{m,\ell})\bigr)^{(\pm)}:=R_m\bigl(E(K_{m,\ell},E)/p^mE(K_{m,\ell})\bigr)^{\pm} \]
and
\[ \bigl(H^1(K_{m,\ell},E)_{p^m}\bigr)^{(\pm)}:=R_m\bigl(H^1(K_{m,\ell},E)_{p^m}\bigr)^\pm, \]
where the superscript $\pm$ on the right denotes the submodule on which complex conjugation acts as $\pm$. 

\begin{lemma} \label{local-splitting-lemma}
Let $\ell$ be a Kolyvagin prime for $p^m$.
\begin{enumerate}
\item The $\tilde R_m$-module $\bigl(E(K_{m,\ell})/p^mE(K_{m,\ell})\bigr)^{(\pm)}$ is isomorphic to $R_m^{(\pm)}$, hence there is a decomposition $E(K_{m,\ell})/p^mE(K_{m,\ell})\simeq R_m^{(+)}\oplus R_m^{(-)}$.
\item The $\tilde R_m$-module $\bigl(H^1(K_{m,\ell},E)_{p^m}\bigr)^{(\pm)}$ is isomorphic to $R_m^{(\pm)}$, hence there is a decomposition $H^1(K_{m,\ell},E)_{p^m}\simeq R_m^{(+)}\oplus R_m^{(-)}$.
\end{enumerate}
\end{lemma}

\begin{proof} Part (1) is \cite[\S 1.2, Lemma 4]{Ber1}, while part (2) is \cite[\S 1.2, Corollary 6]{Ber1}. \end{proof}

\subsection{Action of complex conjugation} \label{sec-complex-conj}

In this subsection we study the action of $\Gal(K/\Q)$ on Selmer groups. These results will be used in \S \ref{sec-families} to show the existence of suitable families of Kolyvagin primes. 

The canonical action of $\tau$ on $\mathcal X_\infty^\pm$ makes it into a $\tilde\Lambda$-module. Recall the element $x\in \mathcal X_\infty^\pm$ chosen at the beginning of this section such that $\pi^\pm  (x)\neq0$ and $\tau(x)=\epsilon x$ for some $\epsilon\in\{\pm\}$. Now pick an element $y\in\ker(\pi^\pm )^{-\epsilon}$ and consider the surjection of $\tilde\Lambda$-modules
\[ \Lambda^{(\epsilon)}\oplus\Lambda^{(-\epsilon)}\longepi\Lambda x\oplus\Lambda y\subset\mathcal X_\infty^\pm   \oplus\mathcal X_\infty^\pm   ,\qquad(\xi,\eta)\longmapsto(\xi x,\eta y). \] Since $\mathcal H_\infty^\pm$ is torsion-free by Proposition \ref{H-rank-prop}, $\ker(\pi^\pm )\cap \Lambda x=\{0\}$, hence $\Lambda x\cap\Lambda y=\{0\}$. Therefore the canonical map of $\tilde\Lambda$-modules $\Lambda x\oplus\Lambda y\rightarrow \mathcal X_\infty^\pm   $ given by the sum is injective. Composing the last two maps, we get a map of $\tilde\Lambda$-modules 
\begin{equation} \label{theta-eq}
\vartheta^\pm:\Lambda^{(\epsilon)}\oplus\Lambda^{(-\epsilon)}\longrightarrow\mathcal X_\infty^\pm   
\end{equation}
that sends $(\alpha,\beta)$ to $\alpha x+\beta y$.

By Lemma \ref{inj-selmer-lemma}, there is a canonical injection 
\begin{equation}\label{eq5}
\Sel^\pm _{p^m}(E/K_m)\longmono\Sel^\pm _{p^\infty}(E/K_m).
\end{equation}
Let 
\[ \mathcal Z^\pm_m:=\Hom_{\Z_p}\bigl(\Sel^\pm _{p^m}(E/K_m),\Q_p/\Z_p\bigr) \]
be the Pontryagin dual of $\Sel^\pm _{p^m}(E/K_m)$. One may then consider the surjection of compact $\Lambda$-modules  
\begin{equation} \label{p_m-eq}
p_m^\pm:\mathcal X_\infty^\pm\longepi\mathcal X_\infty^\pm/\omega_m^\pm\mathcal X_\infty^\pm\longepi \mathcal X^\pm_m/\omega_m^\pm\mathcal X^\pm_m\longepi\mathcal Z^\pm_m/\omega_m^\pm \mathcal Z^\pm_m,
\end{equation}
where the first arrow is the canonical projection, the second is \eqref{coro-CT} and the third is obtained from \eqref{eq5} by Pontryagin duality. 

%Since, as above, $\delta$ is fixed once and for all, we put \[\mathscr{R}_m:=R_m^\delta;\quad \tilde R^\pm _m:=\tilde R_m^\delta; \quad R_m^{(\pm)}:=R_m^{\delta,(\pm)}.\]

Let us define the following $R_m^\pm$-submodules of $\mathcal Z^\pm _m /\omega_m^\pm\mathcal Z^\pm_m$: 
\[ \begin{split}
   Z^\pm_m&:=\bigl((p^\pm_m\circ\vartheta)(\Lambda^{(\epsilon)}\oplus\{0\}))\bigr)\cap\big((p^\pm_m\circ\vartheta)(\{0\}\oplus \Lambda^{(-\epsilon)})\big),\\
  W_m^{\pm,(\epsilon)}&:=\big((p^\pm_m\circ\vartheta)(\Lambda^{(\epsilon)}\oplus\{0\})\big)/Z_m,\\
  W_m^{\pm,(-\epsilon)}&:= \big((p^\pm_m\circ\vartheta)(\{0\}\oplus \Lambda^{(-\epsilon)})\big)/Z_m.
   \end{split} \]
Set
\[ \Sigma^\pm_m:=\bigl((\mathcal Z^\pm _m /\omega_m^\pm\mathcal Z^\pm _m)/Z^\pm_m\bigr)^{\!\vee}.\] 
The submodule $Z^\pm_m$ being closed in $\mathcal Z^\pm _m /\omega_m^\pm\mathcal Z^\pm_m $, Pontryagin duality yields a natural injection $\Sigma^\pm_m\hookrightarrow \Sel_{p^m}^\pm(E/K_m)^{\omega_m^\pm=0}$ of $R^\pm _m $-modules. We obtain a chain of maps of $\tilde\Lambda$-modules 
\begin{equation} \label{chain-eq}
\Lambda^{(\epsilon)}\oplus\Lambda^{(-\epsilon)}\longepi W_m^{\pm,(\epsilon)}\oplus W_m^{\pm,(-\epsilon)}\longmono(\Sigma_m^\pm)^\vee
\end{equation}
in which the surjection is induced by $p_m\circ\vartheta$ and the injection is given by the sum of the components. By construction, the composition in \eqref{chain-eq} factors through the surjection 
\[ \Lambda^{(\epsilon)}\oplus\Lambda^{(-\epsilon)}\longepi\Lambda_m^{(\epsilon)}\oplus\Lambda_m^{(-\epsilon)}\longepi(R^\pm_m)^{ (\epsilon)}\oplus(R^\pm_m)^{(-\epsilon)}. \] 
Write
\[ \vartheta^\pm_m:(R^\pm_m)^{(\epsilon)}\oplus(R^\pm_m)^{(-\epsilon)}\longepi W_m^{\pm,(\epsilon)}\oplus W_m^{\pm,(-\epsilon)}\longmono(\Sigma_m^\pm)^\vee \]
for the resulting map of $\tilde R_m^\pm$-modules; if $\bar x$ and $\bar y$ denote the images of $x$ and $y$ in $(\Sigma_m^\pm)^\vee$ then $\vartheta^\pm_m((\alpha,\beta))=\alpha\bar x+\beta\bar y$. 

\begin{lemma} \label{isom-lemma}
There is an isomorphism of $R_m^\pm$-modules $(R_m^\pm)^\vee\simeq\omega_m^\mp R_m$.
\end{lemma}

\begin{proof} Since $R_m=\Z/p^m\Z[G_m]$, there is a canonical isomorphism
\[ \Phi:R_m^\vee\overset\simeq\longrightarrow R_m \]
of $R_m$-modules (hence of $\Lambda$-modules). The surjection $R_m\twoheadrightarrow R_m^\pm$ induces, by duality, an injection $i:(R_m^\pm)^\vee\hookrightarrow R_m^\vee$, and we obtain an injection $\Phi\circ i:(R_m^\pm)^\vee\hookrightarrow R_m$. The image of $\Phi\circ i$ is annihilated by $\omega_m^\pm$, hence $\Phi\circ i$ induces an injection of $R_m$-modules 
\[ \Psi:(R_m^\pm)^\vee\longmono\tilde\omega_m^\mp R_m.\] 
Now pick $x\in R_m$ and consider $\tilde\omega_m^\mp x\in\tilde\omega_m^\mp R_m$. If $\varphi_x\in R_m^\vee$ is such that $\Phi(\varphi_x)=x$ then $\Phi(\tilde\omega_m^\mp \varphi_x)=\tilde\omega_m^\mp x$. In order to show that $\Psi$ is surjective we need to check that $\tilde\omega_m^\mp \varphi_x$ factors through $R_m^\pm$. But this is clear: $(\tilde\omega_m^\mp\varphi_x)(\omega_m^\pm \lambda)=\varphi_x(\omega_m^\pm\tilde\omega_m^\mp\lambda)=0$ for all $\lambda\in R_m$ because $\omega_m=\omega_m^\pm\tilde\omega_m^\mp$ kills $R_m$. \end{proof}

Taking the $\Gal(K/\Q)$-action into account, Lemma \ref{isom-lemma} yields isomorphisms $\bigl((R_m^\pm)^\vee\bigr)^{(\pm\epsilon)}\simeq(\omega_m^\mp R_m)^{(\pm\epsilon)}$ of $\tilde R_m^\pm$-modules. Furthermore, $\bigl((R_m^\pm)^\vee\bigr)^{(\pm\epsilon)}=\bigl((R_m^\pm)^{(\pm\epsilon)}\bigr)^\vee$ and $(\omega_m^\mp R_m)^{(\pm\epsilon)}\simeq (R_m^\pm)^{(\pm\epsilon)}$ under the isomorphism $\omega_m^\mp R_m\simeq R_m^\pm$. Composing these isomorphisms, we get an isomorphism of $\tilde R_m^\pm$-modules 
\begin{equation} \label{i_m}
i_m^{\pm,(\pm\epsilon)}:\left((R_m^\pm)^{(\pm\epsilon)}\right)^\vee \overset\simeq\longrightarrow  
(\tilde\omega_m^\mp R_m)^{(\pm\epsilon)}\overset{\simeq}\longrightarrow(R_m^\pm)^{(\pm\epsilon)}.
\end{equation}  
Set $i^\pm_m:=i_m^{\pm,(\epsilon)}\oplus i_m^{\pm,(-\epsilon)}$. Composing the Pontryagin dual $(\vartheta^\pm_m)^\vee$ of $\vartheta^\pm_m$ with $i^\pm_m$, we get a map of $\tilde R^\pm _m$-modules that we still denote by  
\begin{equation} \label{theta_m-dual}
(\vartheta^\pm_m)^\vee:\Sigma^\pm_m\longrightarrow  ( R^\pm_m)^{(\epsilon)}\oplus(R^\pm_m)^{(-\epsilon)}. 
\end{equation}
If $\overline{\Sigma}^\pm_m:=\Sigma^\pm_m\big/\ker\bigl((\vartheta^\pm_m)^\vee\bigr)$ then there is an injection $(\bar\vartheta^\pm_m)^\vee:\overline{\Sigma}^\pm_m\longmono(R^\pm_m)^{(\epsilon)}\oplus(R^\pm_m)^{(-\epsilon)}$ of $\tilde R^\pm _m$-modules. Define 
\[ \overline{\Sigma}_m^{\pm,(\epsilon)}:=\bigl((\bar{\vartheta}^\pm_m)^{\vee}\bigr)^{-1}\bigl( ( R^\pm_m)^{(\epsilon)}\oplus\{0\}\bigr),\qquad\overline{\Sigma}_m^{\pm,(-\epsilon)}:=\bigl((\bar{\vartheta}^\pm_m)^{\vee}\bigr)^{-1}\bigl(\{0\}\oplus(R^\pm_m)^{(-\epsilon)}\bigr). \] 
Then there is a splitting 
\begin{equation} \label{eq-Sigma_m}
\overline{\Sigma}^\pm_m=\overline{\Sigma}_m^{\pm,(\epsilon)}\oplus\overline{\Sigma}_m^{\pm,(-\epsilon)}
\end{equation}
of $\tilde R^\pm_m$-modules. Taking $G_m$-invariants, we obtain an injection 
\[ \bigl(\overline{\Sigma}_m^{\pm,(\pm\epsilon)}\bigr)^{G_m}\longmono\bigl( ( R^\pm_m)^{(\pm\epsilon)}\bigr)^{G_m}\simeq\Z/p^m\Z \]
of $\Z /p^m\Z$-modules, hence $\bigl(\overline{\Sigma}_m^{\pm,(\pm\epsilon)}\bigr)^{G_m}$ is isomorphic to $\Z/p^{m^{\pm,(\pm\epsilon)}}\Z$ for a suitable integer $0\leq m^{\pm,(\pm\epsilon)}\leq m$ (of course, nothing prevents $\bigl(\overline{\Sigma}_m^{\pm,(\pm\epsilon)}\bigr)^{G_m}$ from being trivial).   

\subsection{Compatibility of the maps} \label{sec-duals} 

In order to ensure compatibility of the various maps appearing in the previous subsection as $m$ varies, in the sequel it will be useful to make a convenient choice of the isomorphism $i_m^{\pm,(\epsilon)}$ introduced in \eqref{i_m}.  
 
Let 
\[ \pi^\pm_m:\mathcal Z^\pm _m/\omega_m\mathcal Z^\pm_m\longepi\mathcal H_m^\pm:=\Hom_{\Z_p}(\E^\pm _m,\Q_p/\Z_p) \] 
denote the dual of the inclusion $\E^\pm_m\subset\Sel_{p^m}^\pm (E/K_m)^{\omega_m^\pm =0}$. Since $y\in\ker(\pi^\pm)$, we have $\pi^\pm_m(Z^\pm_m)=0$, hence there is a surjection $\bar\pi^\pm_m:(\Sigma^\pm_m)^\vee\twoheadrightarrow\mathcal H_m^\pm$ showing, by duality, that $\E^\pm_m$ is actually a submodule of $\Sigma^\pm_m$. Since $(\bar\pi^\pm_m\circ\vartheta^\pm_m)\bigl(\{0\}\oplus(R_m^\pm)^{(-\epsilon)}\bigr)=\{0\}$, again because $y\in \ker(\pi^\pm)$, the dual of $\bar\pi^\pm_m\circ\vartheta^\pm_m$ factors through a map 
\[ \tilde\psi^\pm_m:\E^\pm _m\longrightarrow\bigl((R_m^\pm)^{(\epsilon)}\bigr)^\vee. \] 

\begin{proposition} \label{prop:compatibilities} 
One can choose the isomorphisms $i_m^{\pm,(\epsilon)}$ in \eqref{i_m} so that if $\psi^\pm_m$ denotes the composition 
\[ \psi^\pm_m:\E^\pm _m\xrightarrow{\tilde\psi^\pm_m}\bigl((R_m^\pm)^{(\epsilon)}\bigr)^\vee\xrightarrow{i_m^{\pm,(\epsilon)}}(R_m^\pm)^{(\epsilon)} \]
then the $R^\pm_m$-modules $\psi^\pm_m(\E^\pm_m)$ are generated by elements $\theta^\pm_m\in R^\pm_m$ satisfying
\[ \theta^\pm_\infty:={(\theta^\pm_m)}_{m\geq1}\in\varprojlim_m R^\pm_m\simeq\Lambda. \] 
\end{proposition}

\begin{proof} Here we consider only the case of sign $+$, the other case being similar. 
Since the statement is independent of the $\Gal(K/\Q)$-action (all maps are equivariant for this action), we ignore it. For each $m\geq1$ fix a generator $\theta_m^+$ of $\psi_m^+(\E_m^+)$ and use the shorthand ``$\cores$'' for the corestriction map $\cores_{K_{m+1}/K_m}$. First suppose that $m$ is odd. In this case $\tilde\omega_m^-=\tilde\omega_{m+1}^-$ and $\cores(\alpha_{m+1}^+)=-\alpha_{m}^+$. There is a commutative diagram 
\[ \xymatrix@C=35pt{\mathcal E_{m+1}^+\ar[r]^-{\tilde\psi^+_{m+1}}\ar[d]^-{\cores}& \left(R_{m+1}^+\right)^\vee  \ar[r]^-\simeq 
& \tilde \omega_{m+1}^- R_{m+1}\ar@{->>}[d]
\ar[r]^-\simeq & R_{m+1}^+\ar@{->>}[d]
 \\
\mathcal E_m^+\ar[r]^-{\tilde\psi^+_{m}}&  \left(R_m^+\right)^\vee  \ar[r]^-\simeq 
& \tilde \omega_m^- R_m
\ar[r]^-\simeq & R_m^+} \]  
where the vertical unadorned arrows are projections, and replacing $i_m^{+,(\epsilon)}$ with $u_mi_m^{+,(\epsilon)}$ for a suitable unit $u_m$ 
gives the compatibility of $\theta^+_{m+1}$ and $\theta^+_m$ under projection. Now suppose that $m$ is even. In this case $\tilde\omega_{m+1}^-\equiv p\tilde\omega_m^-$ in $\Lambda_m$ and $\cores(\alpha_{m+1}^+)=p\alpha_m^+$. Therefore there is a commutative diagram
\[ \xymatrix@C=35pt{\mathcal E_{m+1}^+\ar[r]^-{\tilde\psi_{m+1}}& \left(R_{m+1}^+\right)^\vee  \ar[r]^-\simeq & \tilde \omega_{m+1}^- R_{m+1}\ar@{->>}[d]^-{1/p}
\ar[r]^-\simeq & R_{m+1}^+\ar@{->>}[d]\\ \mathcal E_m^+\ar[r]^-{\tilde\psi_{m}}&  \left(R_m^+\right)^\vee  \ar[r]^-\simeq 
& \tilde \omega_m^- R_m
\ar[r]^-\simeq & R_m^+} \] 
that again shows the compatibility between $\theta^+_{m+1}$ and $\theta^+_m$. The result follows. \end{proof}

From now on, fix the isomorphisms $i_m^{\pm,(\epsilon)}$ as in Proposition \ref{prop:compatibilities}, so that $\theta^\pm_\infty\in\Lambda$.  In the following, we will implicitly identify $(R_m^\pm)^{(\epsilon)}$ and its Pontryagin dual by means of the above maps. We will also identify $(R_m^\pm)^{(-\epsilon)}$ with its Pontryagin dual, but we will not need to specify a convenient isomorphism in this case. 

\subsection{Galois extensions} \label{galois-subsec}

We introduce several Galois extensions attached to the modules defined in \S \ref{sec-complex-conj}; in doing this, we follow \cite[\S 1.3]{Ber1} closely. We start with a discussion of a general nature. 

For any $\Z/p^m\Z$-submodule $S\subset\Sel_{p^m}(E/K_m)$ we define the extension $M_S$ of $K_m(E_{p^m})$ cut out by $S$ as follows. Set 
\[ \mathcal G_m:=\Gal\bigl(K_m(E_{p^m})/K_m\bigr). \] 
With a slight abuse, we shall often view $\mathcal G_m$ as a subgroup of $\GL_2(\Z/p^m\Z)$, according to convenience. By \cite[\S 1.3, Lemma 2]{Ber1}, whose proof does not use the ordinariness of $E$ at $p$ assumed in \emph{loc. cit.}, there is an isomorphism 
\[ \mathcal G_m\simeq\GL_2(\Z/p^m\Z). \]
By \cite[\S 1.3, Lemma 1]{Ber1}, whose proof works in our case too, restriction gives an injection 
\[ \Sel_{p^m}(E/K_m)\longmono \Sel_{p^m}\bigl(E/K_m(E_{p^m})\bigr)^{\mathcal G_m}. \]
Define $G_{K_m(E_{p^m})}^{\rm ab}:=\Gal\bigl(K_m(E_{p^m})^{\rm ab}/K_m(E_{p^m})\bigr)$ where $K_m(E_{p^m})^{\rm ab}$ is the maximal abelian extension of $K_m(E_{p^m})$. It follows that there is an identification 
\[ H^1\bigl(K_m(E_{p^m}),E_{p^m}\bigr)^{\mathcal G_m}=\Hom_{\mathcal G_m}\bigl(G_{K_m(E_{p^m})}^{\rm ab},E_{p^m}\bigr) \] 
of $\Z/p^m\Z$-modules, where $\Hom_{\mathcal G_m}(\bullet,\star)$ stands for the group of $\mathcal G_m$-homomorphisms. Thus we obtain an injection of $\Z/p^m\Z$-modules 
\begin{equation} \label{eq13}
S\longmono\Hom_{\mathcal G_m}\bigl(G_{K_m(E_{p^m})}^{\rm ab},E_{p^m}\bigr),\qquad s\longmapsto\varphi_s,
\end{equation}
and for every $s\in S$ we let $M_s$ denote the subfield of $K_m(E_{p^m})^{\rm ab}$ fixed by $\ker(\varphi_s)$. In other words, $M_s$ is the smallest abelian extension of 
${K_m(E_{p^m})}$ such that the restriction of $\varphi_s$ to $\Gal(K_m(E_{p^m})^{\rm ab}/M_s)$ is trivial. The maps $\varphi_s$ induce injections  
\[ \varphi_s:\Gal\bigl(M_s/K_m(E_{p^m})\bigr)\longmono E_{p^m} \]
of $\mathcal G_m$-modules. Let $M_S\subset K_m(E_{p^m})^{\rm ab}$ denote the composite of all the fields $M_s$ for $s\in S$.

By \cite[Lemma 3 p. 159]{Ber1}, the map 
\begin{equation} \label{eq14}
\Gal\bigl(M_S/K_m(E_{p^m})\bigr)\longrightarrow\Hom(S,E_{p^m}),\qquad g\longmapsto\bigl(s\mapsto\varphi_s({g|}_{M_s})\bigr) 
\end{equation}
is a $\mathcal G_m$-isomorphism and \eqref{eq13} induces an isomorphism 
\[ S\overset\simeq\longrightarrow\Hom_{\mathcal G_m}\bigl(\Gal(M_S/K_m(E_{p^m})),E_{p^m}\bigr) \]
of $\Z/p^m\Z$-modules;  here $\Hom(\bullet,\star)$ is a shorthand for $\Hom_{\Z/p^m\Z}(\bullet,\star)$. One can show that, given two subgroups $S'\subset S\subset \Sel_{p^m}(E/K_m)$,  there is a canonical isomorphism of groups 
\begin{equation} \label{S'-prop}
\Gal(M_S/M_{S'})\simeq\Hom(S/S',E_{p^m})
\end{equation}
and, conversely, for every subgroup $\bar S$ of $S/S'$ there is a subextension $M_{\bar S}/M_{S'}$ of $M_S/M_{S'}$ such that 
\begin{equation} \label{S'-prop-bis}
\Gal(M_{\bar S}/M_{S'})\simeq\Hom(\bar S,E_{p^m}). 
\end{equation}
In this case, we say that $M_S/M_{S'}$ is the extension associated with the quotient $S/S'$. 

Now we apply these constructions to the setting of \S \ref{sec-complex-conj}. To simplify the notation, put 
\[ \Sel^\pm_m:=\Sel_{p^m}^\pm(E/K_m)^{\omega_m^\pm =0},\qquad\Sel_\infty^\pm:=\varinjlim_m\Sel^\pm_m. \]
Let $M^\pm_m$ denote the field cut out by the subgroup $\Sel^\pm_m$; then $M^\pm_m\subset M^\pm_{m+1}$. By construction, there are canonical surjections
\begin{equation} \label{surj-gal-prop}
\Gal\bigl(M^\pm_{m+1}/K_{m+1}(E_{p^{m+1}})\bigr)\longepi\Gal\bigl(M^\pm_m/K_m(E_{p^m})\bigr).
\end{equation}
Define
\[ M^\pm_\infty:=\varinjlim_mM^\pm_m,\qquad K_\infty(E_{p^\infty}):=\varinjlim_m K_m(E_{p^m}), \]
so that 
\[ \Gal\bigl(M^\pm_\infty/K_\infty(E_{p^\infty})\bigr)=\invlim_m\Gal\bigl(M^\pm_m/K_m(E_{p^m})\bigr), \]
the inverse limit being taken with respect to the maps in \eqref{surj-gal-prop}. By \eqref{eq14}, for every $m\geq0$ there is an isomorphism 
\[ \Gal\bigl(M^\pm_m/K_m(E_{p^m})\bigr)\simeq\Hom\bigl(\Sel^\pm_m,E_{p^m}\bigr) \] 
of $\Z/p^m\Z[\mathcal G_m]$-modules, hence there is an isomorphism of $\Z_p[\![\mathcal G_\infty]\!]$-modules  
\[ \Gal\bigl(M^\pm_\infty/K_\infty(E_{p^\infty})\bigr)\simeq\Hom\bigl(\Sel^\pm_\infty,E_{p^\infty}\bigr), \] 
where $\Z_p[\![\mathcal G_\infty]\!]:=\sideset{}{_m}\invlim\Z_p[\mathcal G_m]$ is defined with respect to the canonical maps $\mathcal G_{m+1}\rightarrow\mathcal G_m$. 

Now recall the map $(\vartheta^\pm_m)^\vee$ of \eqref{theta_m-dual} and let $L^\pm_m\subset M^\pm_m$ be the extension of $K_m(E_{p^m})$ cut out by $\ker\bigl((\vartheta^\pm_m)^\vee\bigr)$. Then there are canonical $\mathcal G_m$-isomorphisms
\[ \Gal\bigl(L^\pm_m/K_m(E_{p^m})\bigr)\simeq\Hom\bigl(\ker\bigl((\vartheta^\pm_m)^\vee\bigr),E_{p^m}\bigr) \]  
and  
\begin{equation} \label{L_m-isom2-eq}
\Gal(M^\pm_m/L^\pm_m)\simeq\Hom\bigl(\overline{\Sigma}^\pm_m,E_{p^m}\bigr)\simeq\Hom\!\Big(\overline{\Sigma}_m^{\pm,(\epsilon)},E_{p^m}\Big)\oplus\Hom\!\Big(\overline{\Sigma}_m^{\pm,(-\epsilon)},E_{p^m}\Big);
\end{equation}
here \eqref{L_m-isom2-eq} is a consequence of \eqref{S'-prop-bis}. Moreover, write $L_m^{\pm,(\pm\epsilon)}$ for the subextension of $M^\pm_m/L^\pm_m$ corresponding to $\overline{\Sigma}_m^{\pm,(\pm\epsilon)}$ (cf. \eqref{S'-prop}); then $L_m^{\pm,(\epsilon)}\cap L_m^{\pm,(-\epsilon)}=L^\pm_m$ and $M^\pm_m=L_m^{\pm,(\epsilon)}\cdot L_m^{\pm,(-\epsilon)}$. Finally, let $\tilde L_m^{(\pm\epsilon)}$ denote the extension of $L_m^{\pm,(\pm\epsilon)}$ corresponding to $\bigl(\overline{\Sigma}_m^{\pm,(\pm\epsilon)}\bigr)^{G_m}$. We have $\tilde L_m^{\pm,(+)}\cap\tilde L_m^{\pm,(-)}=L^\pm_m$ and  
\[ \Gal\bigl(\tilde L_m^{\pm,(\pm)}/L^\pm_m\bigr)\simeq\Hom\!\Big(\bigl(\overline{\Sigma}_m^{\pm,(\pm)}\bigr)^{G_m},E_{p^m}\Big). \]
Furthermore, if $\tilde L^\pm_m:=\tilde L_m^{\pm,(+)}\cdot\tilde L_m^{\pm,(-)}$ then 
\[ \begin{split}
   \Gal\bigl(\tilde L^\pm_m/L^\pm_m\bigr)&\simeq\Hom\!\Big(\bigl(\overline{\Sigma}_m^{\pm,(+)}\bigr)^{G_m},E_{p^m}\Big)\oplus\Hom\!\Big(\bigl(\overline{\Sigma}_m^{\pm,(-)}\bigr)^{G_m},E_{p^m}\Big)\\&\simeq\Hom\!\Big(\bigl(\overline{\Sigma}^\pm_m\bigr)^{G_m},E_{p^m}\Big),
   \end{split} \]
where the second isomorphism follows by taking $G_m$-invariants in \eqref{eq-Sigma_m}. 
Since, by Lemma \ref{inj-selmer-lemma}, $\Sel^\pm_m$ injects via restriction into $\Sel_{p^{m+1}}\bigl(E/K_{m+1}(E_{p^{m+1}})\bigr)$, restriction induces an injection $\bigl(\overline{\Sigma}^\pm_m\bigr)^{G_m}\hookrightarrow\bigl(\overline{\Sigma}^\pm_{m+1}\bigr)^{G_{m+1}}$. It follows that for every $m\geq0$ there is a canonical projection 
\begin{equation} \label{cheb2}
\Gal\bigl(\tilde L^\pm_{m+1}/L^\pm_{m+1}\bigr)\longepi\Gal\bigl(\tilde L^\pm_m/L^\pm_m\bigr).
\end{equation} 
To introduce the last field extensions we need, we dualize the exact sequence
\[ (R^\pm_m)^{(\epsilon)}\oplus(R^\pm_m)^{(-\epsilon)}\xrightarrow{\vartheta^\pm_m}(\Sigma^\pm_m)^\vee\longrightarrow(\Sigma^\pm_m)^\vee/\mathrm{im}(\vartheta_m^\pm)\longrightarrow0 \] 
and get an isomorphism $\ker\bigl((\vartheta^\pm_m)^\vee\bigr)\simeq\bigl((\Sigma^\pm_m)^\vee/\mathrm{im}(\vartheta^\pm_m)\bigr)^{\!\vee}$. Moreover, dualizing 
\[ 0\longrightarrow\mathrm{im}(\vartheta^\pm_m)\longrightarrow(\Sigma^\pm_m)^\vee\longrightarrow\ker\bigl((\vartheta^\pm_m)^\vee\bigr)^\vee\longrightarrow0 \] 
gives a short exact sequence
\begin{equation} \label{dual-seq-eq}
0\longrightarrow\ker\bigl((\vartheta^\pm_m)^\vee\bigr)\longrightarrow\Sigma^\pm_m\xrightarrow{(\vartheta^\pm_m)^\vee}\mathrm{im}(\vartheta^\pm_m)^\vee\longrightarrow0.
\end{equation}
Finally, with maps $\vartheta^\pm$ and $p^\pm_m$ defined as in \eqref{theta-eq} and \eqref{p_m-eq}, write $U^\pm_m$ for the $\tilde R^\pm _m$-submodule of $\Sel^\pm_m$ such that there is an indentification 
\begin{equation} \label{image-eq}
\mathcal I^\pm_m:={\rm im}(p^\pm_m\circ\vartheta^\pm)=(\Sel^\pm_m/U^\pm_m)^{\!\vee}. 
\end{equation}
Namely, consider the short exact sequence
\[ 0\longrightarrow\mathcal I^\pm_m\longrightarrow(\Sel^\pm_m)^\vee\longrightarrow(\Sel^\pm_m)^\vee/\mathcal I^\pm_m\longrightarrow0. \]
Since $\mathcal I^\pm_m$ is compact, hence closed in $(\Sel^\pm_m)^\vee$, dualizing the sequence above gives 
\begin{equation} \label{duals-eq}
0\longrightarrow\bigl((\Sel^\pm_m)^\vee/\mathcal I^\pm_m\bigr)^{\!\vee}\longrightarrow\Sel^\pm_m\longrightarrow(\mathcal I^\pm_m)^\vee\longrightarrow0. 
\end{equation}
Now define $U^\pm_m:=\bigl((\Sel^\pm_m)^\vee/\mathcal I^\pm_m\bigr)^{\!\vee}$ and view $U^\pm_m$ as an $\tilde R^\pm_m$-submodule of $\Sel^\pm_m$ via \eqref{duals-eq}. Then there is a natural identification
\begin{equation} \label{quotient-eq}
\Sel^\pm_m/U^\pm_m=(\mathcal I^\pm_m)^\vee,
\end{equation}
and dualizing \eqref{quotient-eq} gives \eqref{image-eq}.

Write $\tilde M^\pm_m$ for the field cut out by $U^\pm_m$. As $p^\pm_m\circ\vartheta^\pm$ factors through $(R^\pm_m)^{(\epsilon)}\oplus(R^\pm_m)^{(-\epsilon)}$, there is a commutative diagram 
\[ \xymatrix{\Lambda^{(\epsilon)}\oplus\Lambda^{(-\epsilon)}\ar@{->>}[d]\ar[r]^-{\vartheta^\pm}&\mathcal X_\infty^\pm\ar[r]^-{p^\pm_m}&(\Sel^\pm_m)^\vee\ar@{->>}[d]\\(R^\pm_m)^{(\epsilon)}\oplus(R^\pm_m)^{(-\epsilon)}\ar[rr]^-{\vartheta^\pm_m}&&(\Sigma^\pm_m)^\vee} \] 
that induces a surjection $\mathcal I^\pm_m\twoheadrightarrow\mathrm{im}(\vartheta^\pm_m)$ and then, by duality, an injection $\mathrm{im}(\vartheta^\pm_m)^\vee\hookrightarrow(\mathcal I^\pm_m)^\vee$. From this we obtain a commutative diagram with exact rows
\begin{equation} \label{d22}
\xymatrix{0\ar[r]&\ker\bigl((\vartheta^\pm_m)^\vee\bigr)\ar[r]\ar@{^(->}[d]&\Sigma^\pm_m\ar[r]\ar@{^(->}[d]&\mathrm{im}(\vartheta^\pm_m)^\vee\ar[r]\ar@{^(->}[d]&0\\0\ar[r]&U^\pm_m\ar[r]&\Sel^\pm_m\ar[r]&(\mathcal I^\pm_m)^\vee\ar[r]&0}
\end{equation} 
whose upper row is \eqref{dual-seq-eq}. Denote by $L_m^{\pm,*}$ the field corresponding to $\Sigma^\pm_m$, so that $\tilde L^\pm_m\subset L_m^{\pm,*}$ by \eqref{S'-prop}. Observe that $\tilde M^\pm_m$ and $\tilde L^\pm_m$ are linearly disjoint over $L^\pm_m$. To check this, note that $\tilde M^\pm_m\cap\tilde L^\pm_m\subset\tilde M^\pm_m\cap L^{\pm,*}_m$ and that the second intersection corresponds to the subgroup $U^\pm_m\cap\Sigma^\pm_m$ inside $\Sel^\pm_m$. But diagram \eqref{d22} shows that $\ker\bigl((\vartheta^\pm_m)^\vee\bigr)=U^\pm_m\cap\Sigma^\pm_m$, hence $\tilde M^\pm_m\cap L^{\pm,*}_m=L^\pm_m$; we conclude that $\tilde M^\pm_m\cap\tilde L^\pm_m=L^\pm_m$. It follows that
\[ \Gal\bigl(\tilde M^\pm_m\cdot\tilde L^\pm_m/\tilde M^\pm_m\bigr)\simeq\Gal\bigl(\tilde L^\pm_m/\tilde M^\pm_m\cap\tilde L^\pm_m\bigr)\simeq\Gal\bigl(\tilde L^\pm_m/L^\pm_m\bigr), \] 
and then the inclusion $\tilde M^\pm_m\cdot\tilde L^\pm_m\subset M^\pm_m$ induces a surjection
\begin{equation} \label{gal-onto-eq}
\Gal\bigl(M^\pm_m/\tilde M^\pm_m\bigr)\longepi\Gal\bigl(\tilde L^\pm_m/L^\pm_m\bigr). 
\end{equation}
It follows that for every $m\geq0$ there is a commutative square of surjective maps 
\begin{equation} \label{diagram}
\xymatrix{\Gal\bigl(M^\pm_{m+1}/\tilde M^\pm_{m+1}\bigr)\ar@{->>}[r]\ar@{->>}[d]&\Gal\bigl(\tilde L^\pm_{m+1}/L^\pm_{m+1}\bigr)\ar@{->>}[d]\\\Gal\bigl(M^\pm_m/\tilde M^\pm_m\bigr)\ar@{->>}[r]&\Gal\bigl(\tilde L^\pm_m/L^\pm_m\bigr)}
\end{equation} 
where the horizontal arrows are given by \eqref{gal-onto-eq} and the right vertical arrow is given by \eqref{cheb2}. One easily checks the surjectivity of the left vertical map and the commutativity of \eqref{diagram}. 

\subsection{Families of Kolyvagin primes} \label{sec-families} 

The purpose of this subsection is to show that one can manufacture a Galois-compatible sequence ${(\ell^\pm_m)}_{m\geq1}$ of Kolyvagin primes. More precisely, our goal is to prove  

\begin{proposition}\label{prop-cheb}  
There is a sequence $\ell^\pm_\infty={(\ell^\pm_m)}_{m\geq1}$ of Kolyvagin primes for $p^m$ satisfying the following conditions: 
\begin{enumerate}
\item $\Frob_{\ell^\pm_m}=[\tau g^\pm_m]$ in $\Gal(M^\pm_m/\Q)$ with $g^\pm_m\in\Gal\bigl(M^\pm_m/K_m(E_{p^m})\bigr)$ such that 
\[ {(g^\pm_m)}_{m\geq1}\in\Gal\bigl(M^\pm_\infty/K_\infty(E_{p^\infty})\bigr); \] 
\item restriction induces an injective group homomorphism 
\[ \overline{\res}_{\ell^\pm_m}:\overline{\Sigma}^\pm_m\longmono E\bigl(K_{m,\ell^\pm_m}\bigr)\big/p^mE\bigl(K_{m,\ell^\pm_m}\bigr); \]
\item $p\nmid(\ell^\pm_m+1)^2-a_{\ell^\pm_m}^2$. 
\end{enumerate}
\end{proposition}

\begin{proof} Notation being as in \S \ref{sec-complex-conj} and \S \ref{galois-subsec}, for each choice of sign $\pm$ pick $h_m^{(\pm)}=h_m^{\pm,(\pm)}\in\Gal\bigl(\tilde L^{\pm,(\pm)}_m/L^\pm_m\bigr)$ such that the period of $\bigl(h^{(\pm)}_m\bigr)^\tau h_m^{(\pm)}$ is $p^{m^{(\pm)}}=p^{m^{\pm,(\pm)}}$. To see the existence of an element with this property, observe that if $h_m^{(\pm)}$ corresponds to the homomorphism $\phi:\bigl(\overline{\Sigma}^{\pm,(\pm)}_m\bigr)^{G_m}\rightarrow E_{p^m}$ then $\bigl(h^{(\pm)}_m\bigr)^\tau h_m^{(\pm)}$ corresponds to $x\mapsto \pm\tau\phi(x)+\phi(x)$. In light of this, to show the existence of such an $h_m^{(\pm)}$ it suffices to choose a $\phi$ that takes a generator of $\bigl(\overline{\Sigma}_m^{\pm,(\pm)}\bigr)^{G_m}$ to an element of order $p^{m^{(\pm)}}$ in $E_{p^m}$. Define $h^\pm_m:=\bigl(h_m^{(+)},h_m^{(-)}\bigr)\in\Gal\bigl(\tilde L^\pm_m/L^\pm_m\bigr)$ and choose the sequence ${(h^\pm_m)}_{m\geq1}$ so that the image of $h^\pm_{m+1}$ via surjection \eqref{cheb2} is $h^\pm_m$. Using diagram \eqref{diagram}, select also a compatible sequence of elements $g^\pm_m\in\Gal\bigl(M^\pm_m/K_m(E_{p^m})\bigr)$ such that the image of $g^\pm_m$ in $\Gal\bigl(\tilde L^\pm_m/L^\pm_m\bigr)$ is $h^\pm_m$. For every integer $m\geq1$ choose a prime number $\ell^\pm_m$ such that 
\begin{equation} \label{choice-of-primes}
\Frob_{\ell^\pm_m}=[\tau g^\pm_m]\quad\text{in $\Gal(M^\pm_m/\Q)$}.
\end{equation}
Clearly, $\ell^\pm_m$ is a Kolyvagin prime and the required compatibility conditions are fulfilled by construction, so (1) is satisfied. To check (2), we must show that the restriction is injective. For this, fix a prime $\mathfrak l^\pm_m$ of $M^\pm_m$ above $\ell^\pm_m$ satisfying $\Frob_{\mathfrak l^\pm_m/\ell^\pm_m}=\tau g^\pm_m$. Then the restriction of $\Frob_{\mathfrak l^\pm_m/\ell^\pm_m}$ to $\Gal(\tilde L^\pm_m/L^\pm_m)$ corresponds to an injective homomorphism 
\[ \phi_{\mathfrak l^\pm_m/\ell^\pm_m}:\bigl(\overline\Sigma^\pm_m\bigr)^{G_m}\longmono E_{p^m} \]
consisting in the evaluation at $\Frob_{\mathfrak l^\pm_m/\ell^\pm_m}$; namely, one has 
\[ \phi_{\mathfrak l^\pm_m/\ell^\pm_m}(s)=s\bigl(\Frob_{\mathfrak l^\pm_m/\ell^\pm_m}\bigr) \] 
for all $s\in\bigl(\overline\Sigma^\pm_m\bigr)^{G_m}$. The choice of $\mathfrak l^\pm_m$ determines a prime $\tilde\lambda^\pm_m$ of $K_m$ above $\ell^\pm_m$, and the completion of $K_m$ at $\tilde\lambda^\pm_m$ is isomorphic to the completion $K_{\lambda^\pm_m}$ of $K$ at the unique prime $\lambda^\pm_m$ of $K$ above $\ell^\pm_m$. It follows that the canonical restriction map 
\begin{equation} \label{r}
\bigl(\overline\Sigma^\pm_m\bigr)^{G_m}\longmono E\bigl(K_{\lambda^\pm_m}\bigr)\big/p^mE\bigl(K_{\lambda^\pm_m}\bigr)
\end{equation}
is injective, since the same is true of its composition with the local Kummer map and the evaluation at Frobenius. Suppose now that $s\in\overline{\Sigma}^\pm_m$ is non-zero and $\overline{\res}_{\ell^\pm_m}(s)=0$. In particular, the submodule $( R^\pm _ms)^{G_m}$ of $\bigl(\overline{\Sigma}^\pm_m\bigr)^{G_m}$ is sent to $0$, via \eqref{r}, in the direct summand $E\bigl(K_{\lambda^\pm_m}\bigr)\big/p^mE\bigl(K_{\lambda^\pm_m}\bigr)$ of $E\bigl(K_{m,\ell^\pm_m}\bigr)\big/p^mE\bigl(K_{m,\ell^\pm_m}\bigr)$ corresponding to $\tilde\lambda^\pm_m$. Up to multiplying $s$ by a suitable power of $p$, we may assume that $s$ is $p$-torsion. Now $ R^\pm _ms$ is a non-trivial $\Z/p\Z$-vector space on which the $p$-group $G_m$ acts. By \cite[Proposition 26]{Serre}, the submodule $( R^\pm _ms)^{G_m}$ is non-trivial, and this contradicts the injectivity of \eqref{r}. Summing up, we have proved that all choices of a sequence $\ell^\pm_\infty={(\ell^\pm_m)}_{m\geq1}$ satisfying \eqref{choice-of-primes} enjoy properties (1) and (2) in the statement of the proposition. 

The finer choice of a sequence $\ell^\pm_\infty$ satisfying (3) as well can be made by arguing as in the proof of \cite[Proposition 12.2, (3)]{Nek}; see the proof of \cite[Proposition 3.26]{LV} for details. \end{proof}

\subsection{Local duality} \label{duality-sec}

The aim of this subsection is to bound the $\Lambda$-rank of $\Lambda x\oplus\Lambda y$ by a $\Lambda$-module $V(\ell^\pm_\infty)$ that surjects onto $\Lambda x\oplus\Lambda y$; as notation suggests, $V(\ell^\pm_\infty)$ depends on the choice of a compatible family $\ell^\pm_\infty={(\ell^\pm_m)}_{m\geq1}$ of Kolyvagin primes as in \S \ref{sec-families}.  

By \cite[Ch. I, Corollary 3.4]{Milne-ADT}, if $F$ is a finite extension of $\Q_p$ then the Tate pairing induces a perfect pairing 
\[ {\langle\cdot,\cdot\rangle}_F:H^1(F,E)_{p^m}\times E(F)/p^mE(F)\longrightarrow \Z/p^m\Z \]
that gives rise to a $\tau$-antiequivariant isomorphism   
\[ \delta_{F}:H^1(F,E)_{p^m}\overset\simeq\longrightarrow\bigl(E(F)/p^mE(F)\bigr)^{\!\vee}. \] 
If $F$ is a number field and $v$ is a finite place of $F$ then we also denote ${\langle\cdot,\cdot\rangle}_{F_v}$ by ${\langle\cdot,\cdot\rangle}_{F,v}$. 

Now let $\ell$ be a Kolyvagin prime for $p^m$ and write $\delta_{m,\lambda}$ as a shorthand for $\delta_{K_m,\lambda}$, where $\lambda$ is a prime of $K_m$ dividing $\ell$. Taking the direct sum of the maps $\delta_{m,\lambda}$ over all the primes $\lambda\,|\,\ell$, we get a $\tau$-antiequivariant isomorphism 
\begin{equation} \label{delta-ell}
\delta_{m,\ell}:H^1(K_{m,\ell},E{)}_{p^m}\overset{\simeq}\longrightarrow\bigl(E(K_{m,\ell})/p^mE(K_{m,\ell})\bigr)^\vee.
\end{equation}
Composing $\delta_{m,\ell}$ with the dual of the restriction $\res_{m,\ell}$ defined in \eqref{restriction-2} and the dual of the inclusion $\Sel^\pm_m\subset\Sel_{p^m}(E/K_m)$, we get a map 
\begin{equation} \label{def-V(ell)}
H^1(K_{m,\ell},E{)}_{p^m}\xrightarrow{\delta_{m,\ell}}\bigl(E(K_{m,\ell})/p^mE(K_{m,\ell})\bigr)^\vee\xrightarrow{\res_{m,\ell}^\vee}\Sel_{p^m}(E/K_m)^\vee\longepi(\Sel^\pm_m)^\vee 
\end{equation}
whose image we denote by $V^\pm_m(\ell)$. By construction, $V^\pm_m(\ell)$ is an $ R^\pm _m$-submodule of $(\Sel^\pm_m)^\vee$. 

\begin{proposition} \label{lemma4.4}
Let $\ell^\pm_\infty$ be a sequence of Kolyvagin primes as in Proposition \ref{prop-cheb}.
\begin{enumerate} 
\item For every $m\geq1$ there is a canonical surjection $V^\pm_m(\ell^\pm_m)\twoheadrightarrow W_m^{\pm,(\epsilon)}\oplus W_m^{\pm,(-\epsilon)}$. 
\item For every $m\geq1$ there is a canonical surjection $V^\pm_{m+1}(\ell^\pm_{m+1})\twoheadrightarrow V^\pm_m(\ell^\pm_m)$. 
\end{enumerate}
\end{proposition}

\begin{proof}  Fix an integer $m\geq1$. Composing the isomorphism $\delta_{m,\ell^\pm_m}$ in \eqref{delta-ell} with the dual of the map $\overline{\res}_{\ell^\pm_m}$ introduced in part (2) of Proposition \ref{prop-cheb}, we get a surjection 
\[ H^1\bigl(K_{m,\ell^\pm_m},E\bigr)_{p^m}\xrightarrow{\delta_{m,\ell^\pm_m}}\Big(E\bigl(K_{m,\ell^\pm_m}\bigr)\big/p^mE\bigl(K_{m,\ell^\pm_m}\bigr)\Big)^{\!\vee}\xrightarrow{\overline{\res}_{m,\ell^\pm_m}^\vee}\bigl(\overline{\Sigma}^\pm_m\bigr)^\vee \] 
that, by definition, factors through $V^\pm_m(\ell^\pm_m)$. Now $\bigl(\overline{\Sigma}^\pm_m\bigr)^\vee\simeq{\rm im}(\vartheta^\pm_m)$, which is isomorphic to $W_m^{\pm,(\epsilon)}\oplus W_m^{\pm,(-\epsilon)}$. Therefore we get an $\tilde  R^\pm _m$-equivariant surjection 
\[ V^\pm_m(\ell^\pm_m)\longepi W_m^{\pm,(\epsilon)}\oplus W_m^{\pm,(-\epsilon)}, \]
which proves part (1). 

As for part (2), let us define the map $V^\pm_{m+1}(\ell^\pm_{m+1})\rightarrow V^\pm_m(\ell^\pm_m)$. Consider the diagram 
\[ \xymatrix{R_{m+1}^{(+)}\oplus R_{m+1}^{(-)}\simeq H^1\bigl(K_{m+1,\ell^\pm_{m+1}},E{\bigr)}_{p^m}\ar@{->>}[r]\ar@{->>}[d]&V^\pm_{m+1}(\ell^\pm_{m+1})\ar@{-->}[d]\ar@{^(->}[r]& 
(\Sel^\pm_{m+1})^\vee\ar@{->>}[d]^{\res^\vee}
%(E(K_{m+1,\ell_{m+1}})/p^mE(K_{m+1,\ell}))^\vee\ar^-{\res^\vee}[d]
\\
R_m^{(+)}\oplus R_m^{(-)}\simeq H^1\bigl(K_{m,\ell^\pm_m},E{\bigr)}_{p^m}\ar@{->>}[r]&V^\pm_m(\ell^\pm_m)\ar@{^(->}[r]& 
(\Sel^\pm_m)^\vee
%(E(K_{m,\ell_m})/p^mE(K_{m,\ell}))^\vee
} \]
in which the left vertical surjection is a consequence of part (2) of Lemma \ref{local-splitting-lemma} and the right vertical arrow is surjective because $\Sel^\pm_m$ injects into $\Sel^\pm_{m+1}$ via the restriction map denoted by ``res'' (see Lemma \ref{inj-selmer-lemma} and \S \ref{control-subsec}). This diagram is commutative by the transfer formula (see, e.g., \cite[Ch. V, (3.8)]{Brown}). In light of this, an easy diagram chasing shows the existence of the desired (dashed) surjective homomorphism. \end{proof}

In the rest of the paper, let $\ell^\pm_\infty={(\ell^\pm_m)}_{m\geq1}$ denote a sequence of Kolyvagin primes as in Proposition \ref{prop-cheb}. It follows from part (2) of Proposition \ref{lemma4.4} that we can define the $\Lambda$-module 
\[ V^\pm(\ell^\pm_\infty):=\invlim_mV^\pm_m(\ell^\pm_m). \] 

\begin{proposition} \label{prop4.5}
There is a surjection 
\[ V^\pm(\ell^\pm_\infty)\longepi\Lambda x\oplus\Lambda y. \]
\end{proposition}

\begin{proof} Taking the inverse limit of the maps in part (1) of Proposition \ref{lemma4.4} gives a surjection 
\begin{equation} \label{eq-surj}
V^\pm(\ell^\pm_\infty)\longepi\invlim_m\bigl(W_m^{\pm,(\epsilon)}\oplus W_m^{\pm,(-\epsilon)}\bigr),
\end{equation} 
where we use the fact that the projective system satisfies the Mittag--Leffler condition, as all the modules involved are finite. On the other hand, with $\mathcal I^\pm_m$ as in \eqref{image-eq}, there is a short exact sequence 
\[ 0\longrightarrow Z^\pm_m\longrightarrow \mathcal I^\pm_m\longrightarrow W_m^{\pm,(\epsilon)}\oplus W_m^{\pm,(-\epsilon)}\longrightarrow 0, \] 
and passing to inverse limits shows that there is a short exact sequence 
\[ 0\longrightarrow\invlim_mZ^\pm_m\longrightarrow\Lambda x\oplus \Lambda y\longrightarrow\invlim_m\bigl(W_m^{\pm,(\epsilon)}\oplus W_m^{\pm,(-\epsilon)}\bigr)\longrightarrow0. \]
Since $\Lambda x\cap\Lambda y=\{0\}$ and $\sideset{}{_m}\invlim Z^\pm_m\subset\Lambda x\cap\Lambda y$, we have $\sideset{}{_m}\invlim Z^\pm_m=0$, therefore $\Lambda x\oplus \Lambda y$ is isomorphic to $\sideset{}{_m}\invlim\bigl(W_m^{\pm,(\epsilon)}\oplus W_m^{\pm,(-\epsilon)}\bigr)$. Combining this with \eqref{eq-surj} gives the result. \end{proof}

\subsection{Kolyvagin classes} \label{kolyvagin-subsec}

We briefly review the construction of Kolyvagin classes attached to Heegner points. 

Let $m\geq0$ be an integer and let $\ell$ be a Kolyvagin prime for $p^m$; in particular, $p^m\,|\,\ell+1$ and $p^m\,|\,a_\ell$. Assume also that $p^{m+1}\nmid \ell+1\pm a_\ell$. Let $H_\ell$ be the ring class field of $K$ of conductor $\ell$. The fields $K_m$ and $H_\ell$ are linearly disjoint over the Hilbert class field $H_1$ of $K$ and $\Gal(H_\ell/H_1)$ is cyclic of order $\ell+1$. Let $H_{m,\ell}^{(p)}$ be the maximal subextension of the composite $K_mH_\ell $ having $p$-power degree over $K_m$ and set $\mathfrak G_\ell:=\Gal\bigl(H_{m,\ell}^{(p)}/K_m\bigr)$. By class field theory, if $n_\ell:=\ord_p(\ell+1)$ then $\mathfrak G_\ell\simeq\Z/p^{n_\ell}\Z$; in particular, $m\,|\,n_\ell$.

As in \cite[\S 3]{Gross} and \cite[\S 2.1]{BD96}, we can define a Heegner point $x_{\ell p^m}\in X_0^D(M)(H_{\ell p^m})$ and, with $\pi_E$ as in \eqref{parametrization}, set $y_{\ell p^m}:=\pi_E(x_{\ell p^m})\in E(H_{\ell p^m})$. Let the integer $d(m)$ be as in \eqref{d(m)-eq}, then take the Galois trace 
\[ \alpha_m(\ell):=\tr_{H_{\ell p^{d(m)}}/H_{m,\ell}^{(p)}}(y_{\ell p^{d(m)}})\in E\bigl(H_{m,\ell}^{(p)}\bigr). \] 
Now fix a generator $\sigma_\ell$ of $\mathfrak G_\ell$ and consider the Kolyvagin derivative operator 
\[ {\bf D}_\ell:=\sum_{i=1}^{p^{n_\ell}-1}i\sigma_\ell^i\in\Z/p^m\Z[\mathfrak G_\ell]. \] 
One has $(\sigma_\ell-1){\bf D}_\ell=-\tr_{H_{m,\ell}^{(p)}/K_m}$, hence 
\[ \bigl[{\bf D}_\ell\bigl(\alpha_m(\ell)\bigr)\bigr]\in\Big(E\bigl(H_{m,\ell}^{(p)}\bigr)\big/{p^m}E\bigl(H_{m,\ell}^{(p)}\bigr)\Big)^{\!\mathfrak G_\ell}, \]
where $[\star]$ denotes the class of an element $\star$ in the relevant quotient group. Now observe that, thanks to condition (2) in Assumption \ref{ass}, $E_p\bigl(H_{m,\ell}^{(p)}\bigr)$ is trivial (cf. \cite[Lemma 4.3]{Gross}). Taking $\mathfrak G_\ell$-cohomology of the $p^m$-multiplication map on $E\bigl(H_{m,\ell}^{(p)}\bigr)$ gives a short exact sequence
\[ 0\longrightarrow E(K_m)/p^mE(K_m)\longrightarrow\Big(E\bigl(H_{m,\ell}^{(p)}\bigr)\big/{p^m}E\bigl(H_{m,\ell}^{(p)}\bigr)\Big)^{\!\mathfrak G_\ell}\longrightarrow H^1\bigl(\mathfrak G_\ell,E(H_{m,\ell}^{(p)})\bigr)_{p^m}\longrightarrow0. \] 
Composing the arrow above with the inflation map gives a map
\begin{equation} \label{res-isom-kol-eq}
\Big(E\big(H_{m,\ell}^{(p)}\big)\big/{p^m}E\big(H_{m,\ell}^{(p)}\big)\Big)^{\!\mathfrak G_\ell}\longrightarrow H^1\bigl(\mathfrak G_\ell,E(H_{m,\ell}^{(p)})\bigr)_{p^m}\longrightarrow H^1(K_m,E)_{p^m}.
\end{equation}

\begin{definition}
The \emph{Kolyvagin class} $d_m(\ell)\in H^1(K_m,E)_{p^m}$ is the class corresponding to $[{\bf D}_\ell(\alpha_m(\ell))]$ under the map \eqref{res-isom-kol-eq}.
\end{definition}

For any Kolyvagin prime $\ell$ for $p^m$ fix a $\tau$-antiequivariant isomorphism of $R_m$-modules 
\begin{equation} \label{phi-eq}
\phi_{m,\ell}:H^1(K_{m,\ell},E)_{p^m}\overset\simeq\longrightarrow E(K_{m,\ell})/p^mE(K_{m,\ell})
\end{equation} 
as in \cite[\S 1.4, Proposition 2]{Ber1}. If $v$ is a place of a number field $F$ and $c\in H^1(F,M)$ for a $G_F$-module $M$, we write $\res_v(c)$ for the restriction (or localization) of $c$ at $v$; if $q$ is a prime number, we write $\res_q(c)$ for the sum of the localizations at the primes of $F$ above $q$. 

For the next result, recall from \S \ref{heegner-trace-subsec} that $z_m:=\tr_{H_{p^{d(m)}}/K_m}\bigl(y_{p^{d(m)}}\bigr)\in E(K_m)$. 

\begin{proposition} \label{prop-d-ell} 
The class $d_m(\ell)$ enjoys the following properties: 
\begin{enumerate}
\item if $v$ is a (finite or infinite) prime of $K_m$ not dividing $\ell$ then $\res_v(d_m(\ell))$ is trivial; 
\item $\phi_{m,\ell}\bigl(\res_\ell(d_m(\ell))\bigr)=[\res_\ell(z_m)]$.  
\end{enumerate}
\end{proposition}

\begin{proof} See \cite[Proposition 6.2]{Gross} or \cite[\S 1.4, Proposition 2]{Ber1}. \end{proof} 

\subsection{Global duality} \label{duality-sec2}

In this subsection we use Kolyvagin classes, combined with global reciprocity laws, to bound the rank of the $\Lambda$-module $V^\pm(\ell^\pm_\infty)$ that was introduced in \S \ref{duality-sec} and surjects onto $\Lambda x\oplus\Lambda y$.

Fix a sequence of Kolyvagin primes $\ell^\pm_\infty={(\ell^\pm_m)}_{m\geq1}$ as in \S\ref{duality-sec}. For every $m\geq1$ define
\[ d^+_m=d^+_m(\ell^+_\infty):=\begin{cases} d_m(\ell^+_m)&\text{if $m$ is even}\\[2mm]d_{m-1}(\ell^+_{m-1})&\text{if $m$ is odd}\end{cases} \]
and
\[ d^-_m=d^-_m(\ell^-_\infty):=\begin{cases}d_{m-1}(\ell^-_{m-1})&\text{if $m$ is even}\\[2mm]d_m({\ell^-_m})&\text{if $m$ is odd.}\end{cases} \]
From now on, in order to ease the notation write
\[ \mu(m):=\begin{cases} m&\text{if (the sign is $+$ and $m$ is even) or (the sign is $-$ and $m$ is odd)}\\[2mm] m-1&\text{if (the sign is $+$ and $m$ is odd) or (the sign is $-$ and $m$ is even).} \end{cases} \] 
With this convention in force, $d^\pm_{\mu(m)}$ belongs to $H^1\bigl(K_{\mu(m)},E\bigr)_{p^{\mu(m)}}$. Now recall that if $F$ is a number field, $s\in H^1(F,E_{p^m})$ and $t\in H^1(F,E)_{p^m}$ then
 \begin{equation} \label{global-duality}
\sum_{v}{\big\langle\res_v(s),\res_v(t)\big\rangle}_{F,v}=0,
\end{equation}
where $v$ ranges over all finite places of $F$ and ${\langle\cdot,\cdot\rangle}_{F,v}$ is the local Tate pairing at $v$. By part (1) of Proposition \ref{prop-d-ell}, the class $d_{\mu(m)}^\pm$ is trivial at all the primes not dividing $\ell^\pm_{\mu(m)}$, hence equality \eqref{global-duality} implies that
\begin{equation} \label{vanishing}
\Big(\delta_{\ell^\pm_{\mu(m)}}\!\circ\res_{\ell^\pm_{\mu(m)}}\Big)\big(d_{\mu(m)}^\pm\big)=0.
\end{equation}
The morphism in \eqref{def-V(ell)} defining $V^\pm(\ell^\pm_{\mu(m)})$ factors as
\[ H^1\bigl(K_{\mu(m),\ell^\pm_{\mu(m)}},E\bigr)_{p^{\mu(m)}}\longepi H^1\bigl(K_{\mu(m),\ell^\pm_{\mu(m)}},E\bigr)_{p^{\mu(m)}}\big/\bigl(\omega_{\mu(m)}^\pm\bigr)\longepi V^\pm(\ell^\pm_{\mu(m)})\subset\bigl(\Sel_{\mu(m)}^\pm\bigr)^\vee \] 
because the target is $\omega_{\mu(m)}^\pm$-torsion. Define $\mathcal D_{\mu(m)}^\pm$ to be the $R_{\mu(m)}^\pm$-module generated by the image $\res_{\ell^\pm_{\mu(m)}}(d_{\mu(m)}^\pm)$ of $d_{\mu(m)}^\pm$ in $H^1\bigl(K_{\mu(m),\ell^\pm_{\mu(m)}},E\bigr)_{p^{\mu(m)}}\big/\bigl(\omega_{\mu(m)}^\pm\bigr)$. The decomposition in part (2) of Lemma \ref{local-splitting-lemma} induces a decomposition
\[ H^1\Big(K_{\mu(m),\ell^\pm_{\mu(m)}},E\Big)_{p^{\mu(m)}}\Big/\bigl(\omega_{\mu(m)}^\pm\bigr)\simeq\bigl(R_{\mu(m)}^\pm\bigr)^{(\epsilon)}\oplus\bigl(R_{\mu(m)}^\pm\bigr)^{(-\epsilon)}. \]
Now we collect two lemmas that will be used in the proof of Proposition \ref{prop-global-duality} below. First of all, recall the map $\psi^\pm_{\mu(m)}:\E_{\mu(m)}^\pm\rightarrow\bigl(R_{\mu(m)}^\pm\bigr)^{(\epsilon)}$ of Proposition \ref{prop:compatibilities}.

\begin{lemma}\label{lemma5.10}
$\res_{\ell^\pm_{\mu(m)}}\bigl(\E_{\mu(m)}^\pm\bigr)\simeq\psi^\pm_{\mu(m)}\bigl(\E_{\mu(m)}^\pm\bigr)=R_{\mu(m)}^\pm\theta^\pm_{\mu(m)}$ as $\tilde R_{\mu(m)}^\pm$-modules.
\end{lemma}

\begin{proof} We know from the discussion in \S\ref{sec-duals} that $\E^\pm_{\mu(m)}$ is a submodule of $\Sigma^\pm_{\mu(m)}$, so there is an injection
\begin{equation} \label{equat1}
\E_{\mu(m)}^\pm\big/\bigl(\E^\pm_{\mu(m)}\cap\ker\bigl((\vartheta^\pm_{\mu(m)})^\vee\bigr)\bigr)\longmono\overline\Sigma^\pm_{\mu(m)} 
\end{equation}
where $\bigl(\vartheta^\pm_{\mu(m)}\bigr)^\vee$ is defined in \eqref{theta_m-dual}. Part (2) of Proposition \ref{prop-cheb} shows that composing \eqref{equat1} with the restriction map $\res_{\ell^\pm_{\mu(m)}}$ produces an injection
\[ \E_{\mu(m)}^\pm\big/\bigl(\E^\pm_{\mu(m)}\cap\ker\bigl((\vartheta^\pm_{\mu(m)})^\vee\bigr)\bigr)\longmono E\bigl(K_{\mu(m),\ell^\pm_{\mu(m)}}\bigr)\big/p^{\mu(m)}E\bigl(K_{\mu(m),\ell^\pm_{\mu(m)}}\bigr) \] 
whose image is equal to $\res_{\ell^\pm_{\mu(m)}}\bigl(\E_{\mu(m)}^\pm\bigr)$. Finally, from the definition of $\psi^\pm_{\mu(m)}$ in Proposition \ref{prop:compatibilities} we see that $\psi^\pm_{\mu(m)}\bigl(\E_{\mu(m)}^\pm\bigr)\simeq\E_{\mu(m)}^\pm\big/\bigl(\E^\pm_{\mu(m)}\cap\ker\bigl((\vartheta^\pm_{\mu(m)})^\vee\bigr)\bigr)$, and we are done. \end{proof}

\begin{lemma} \label{lemma5.11}
\begin{enumerate}
\item $\mathcal D_{\mu(m)}^\pm\simeq\bigl(R_{\mu(m)}^\pm\bigr)^{(-\epsilon)}\theta^\pm_{\mu(m)}$ as $\tilde R_{\mu(m)}^\pm$-modules.
\item $\mathcal D_{\mu(m)}^\pm\cap\Big(H^1\bigl(K_{\mu(m),\ell^\pm_{\mu(m)}},E\bigr)_{p^{\mu(m)}}^{(\epsilon)}\big/\bigl(\omega_{\mu(m)}^\pm\bigr)\oplus\{0\}\Big)=\{0\}$. 
\end{enumerate}
\end{lemma}

\begin{proof} For simplicity set
\[ M:=E\bigl(K_{\mu(m),\ell^\pm_{\mu(m)}}\bigr)\big/p^{\mu(m)}E\bigl(K_{\mu(m),\ell^\pm_{\mu(m)}}\bigr),\qquad H:=H^1\bigl(K_{\mu(m),\ell^\pm_{\mu(m)}},E\bigr)_{p^{\mu(m)}}. \]
We know that $\res_{\ell^\pm_{\mu(m)}}\bigl(\E_{\mu(m)}^\pm\bigr)\subset M^{\omega_{\mu(m)}^\pm=0}$ and that the map $M^{\omega_{\mu(m)}^\pm=0}\rightarrow M\big/\omega_{\mu(m)}^\pm M$ is injective.  On the other hand, the isomorphism $\phi=\phi_{\mu(m),\ell^\pm_{\mu(m)}}$ of \eqref{phi-eq} gives a commutative diagram
\[ \tiny{\xymatrix@C=31pt{
(R_\mu^\pm)^{(\epsilon)}\oplus (R_\mu^\pm)^{(-\epsilon)}\ar@{->>}[d]& \ar[l]_-\simeq
H^{(\epsilon)}/(\omega_\mu^\pm )\oplus H^{(-\epsilon)}/(\omega_\mu^\pm) \ar[r]^-{\phi} &
M^{(-\epsilon)}/(\omega_\mu^\pm)\oplus M^{(\epsilon)}/(\omega_\mu^\pm) \ar[r]^-\simeq  & (R_\mu^\pm)^{(-\epsilon)}\oplus
(R_\mu^\pm)^{(\epsilon)} \ar@{->>}[d]
\\
(R_\mu^\pm)^{(-\epsilon)}\ar[rrr]^-\simeq &&& (R_\mu^\pm)^{(\epsilon)}}} \]
in which we have set $\mu:=\mu(m)$ and the right and left isomorphisms in the top row are a consequence of parts (1) and (2) of Lemma \ref{local-splitting-lemma}, respectively. By part (2) of Proposition \ref{prop-d-ell}, the class of $\res_{\ell^\pm_{\mu(m)}}(d_{\mu(m)}^\pm)$ in $H\big/\bigl(\omega_{\mu(m)}^\pm\bigr)$ is sent to the class of $\big[\res_{\ell^\pm_{\mu(m)}}\bigl(z_{\mu(m)}^\pm\bigr)\bigr]$ in $M\big/\bigl(\omega_{\mu(m)}^\pm\bigr)$. Combining Lemma \ref{lemma5.10} with the diagram above proves (1) and (2) simultaneously. \end{proof}

We are now ready to prove the main result of this subsection.

\begin{proposition} \label{prop-global-duality}
The rank of $V^\pm(\ell^\pm_\infty)$ over $\Lambda$ is at most $1$.
\end{proposition}

\begin{proof} We prove the proposition for sign $+$, the other case being similar. For every $m\geq1$ consider the Kolyvagin class $d^+_{\mu(m)}\in H^1\bigl(K_{\mu(m)},E\bigr)_{p^{\mu(m)}}$ defined above and the submodule $\mathcal D_{\mu(m)}^+\subset H^1\bigl(K_{\mu(m),\ell^+_{\mu(m)}},E\bigr)_{p^{\mu(m)}}^{(\pm\epsilon)}\big/\bigl(\omega_{\mu(m)}^+\bigr)$ generated over $R_{\mu(m)}^+$ by $d_{\mu(m)}^+$. Let $\xi_{\mu(m)}^{(\pm\epsilon)}$ denote generators of $H^1\bigl(K_{\mu(m),\ell^+_{\mu(m)}},E\bigr)_{p^{\mu(m)}}^{(\pm\epsilon)}$ as $\tilde R_{\mu(m)}$-modules. The images $\bar\xi_{\mu(m)}^{(\pm\epsilon)}$ of the classes of $\xi_{\mu(m)}^{(\pm\epsilon)}$ are generators of the quotients $H^1\bigl(K_{\mu(m),\ell^+_{\mu(m)}},E\bigr)_{p^{\mu(m)}}^{(\pm\epsilon)}\big/\bigl(\omega_{\mu(m)}^+\bigr)$ as $\tilde R_{\mu(m)}^+$-modules. The image of the cyclic $R_{\mu(m)}^+$-module $\mathcal D_{\mu(m)}^+$ is then generated by the image of
an element of the form $\eta^{(\epsilon)}\xi_{\mu(m)}^{(\epsilon)}+\eta^{(-\epsilon)}\xi_{\mu(m)}^{(-\epsilon)}$ for suitable $\eta^{(\pm\epsilon)}\in R_{\mu(m)}$, and hence isomorphic to the principal $R_{\mu(m)}^+$-module $R_{\mu(m)}^+\bigl(\eta^{(\epsilon)},\eta^{(-\epsilon)}\bigr)$. Using the isomorphism $R_{\mu(m)}^+\simeq\tilde\omega_{\mu(m)}^-R_{\mu(m)}$ of Lemma \ref{isom-lemma}, we see that 
\[ R_{\mu(m)}^+\bigl(\eta^{(\epsilon)},\eta^{(-\epsilon)}\bigr)\simeq R_{\mu(m)}\tilde\omega_{\mu(m)}^-\bigl(\eta^{(\epsilon)},\eta^{(-\epsilon)}\bigr),\qquad R_{\mu(m)}^+\theta^+_{\mu(m)}\simeq R_{\mu(m)}\tilde\omega_{\mu(m)}^-\theta^+_{\mu(m)}. \]
Applying part (1) of Lemma \ref{lemma5.11}, we obtain an isomorphism 
\[ R_{\mu(m)}\tilde\omega_{\mu(m)}^-\theta^+_{\mu(m)}\simeq R_{\mu(m)}\tilde\omega_{\mu(m)}^-\bigl(\eta^{(\epsilon)},\eta^{(-\epsilon)}\bigr). \]
Now \cite[\S 1.2, Lemma 7]{Ber1} shows that $\tilde\omega_{\mu(m)}^-\theta^+_{\mu(m)}\,|\,\tilde\omega_{\mu(m)}^-\bigl(\eta^{(\epsilon)},\eta^{(-\epsilon)}\bigr)$, so there are $\rho_{\mu(m)},\nu_{\mu(m)}\in R_{\mu(m)}$ such that $\tilde\omega_{\mu(m)}^-\bigl(\eta^{(\epsilon)},\eta^{(-\epsilon)}\bigr)=\tilde\omega_{\mu(m)}^-\theta^+_{\mu(m)}\bigl(\rho_{\mu(m)},\nu_{\mu(m)}\bigr)$. Since $\tilde\omega_{\mu(m)}^-R_{\mu(m)}\simeq R_{\mu(m)}^+$, this implies that $\mathcal D_{\mu(m)}^+$ is generated over $R_{\mu(m)}^+$ by the image of an element of the form $\theta^+_{\mu(m)}\bigl(\rho_{\mu(m)},\nu_{\mu(m)}\bigr)$. By part (2) of Lemma \ref{lemma5.11}, we also know that $\nu_{\mu(m)}\in R_{\mu(m)}^\times$. Let us define 
\[ W_{\mu(m)}:=R_{\mu(m)}^+\Big(\rho_{\mu(m)}\omega_{\mu(m)}^-\bar\xi_{\mu(m)}^{(\epsilon)}+\nu_{\mu(m)}\omega_{\mu(m)}^-\bar\xi_{\mu(m)}^{(-\epsilon)}\Big). \] 
Then
\[ H^1\Big(K_{\mu(m),\ell^+_{\mu(m)}},E\Big)_{p^{\mu(m)}}\Big/\big(\omega_{\mu(m)}^+\big)\simeq H^1\Big(K_{\mu(m),\ell^+_{\mu(m)}},E\Big)_{p^{\mu(m)}}^{\!(\epsilon)}\Big/\big(\omega_{\mu(m)}^+\big)\oplus W_{\mu(m)}, \]
from which we deduce that
\[ \theta^+_{\mu(m)}H^1\Big(K_{\mu(m),\ell^+_{\mu(m)}},E\Big)_{p^{\mu(m)}}\simeq\theta^+_{\mu(m)}H^1\Big(K_{\mu(m),\ell^+_{\mu(m)}},E\Big)_{p^{\mu(m)}}^{\!(\epsilon)}\oplus \res_{\ell^+_{\mu(m)}}\big(R_{\mu(m)}d_{\mu(m)}^+\big). \]
Using \eqref{vanishing} we see that the image of $\res_{\ell^+_{\mu(m)}}\big(d_{\mu(m)}^+\big)$ via \eqref{def-V(ell)} is trivial. Thus we get
\[ \delta_{\mu(m),\ell^+_{\mu(m)}}\Big(\theta^+_{\mu(m)}H^1\Big(K_{\mu(m),\ell^+_{\mu(m)}},E\Big)_{p^{\mu(m)}}\Big)\simeq\delta_{\mu(m),\ell^+_{\mu(m)}}\Big(\theta^+_{\mu(m)}H^1\Big(K_{\mu(m),\ell^+_{\mu(m)}},E\Big)_{p^{\mu(m)}}^{\!(\epsilon)}\Big). \] 
Therefore, recalling the definition of $V^+\big(\ell^+_{\mu(m)}\big)$, we conclude that there is an isomorphism of $R_{\mu(m)}^+$-modules
\[ \theta^+_{\mu(m)}V^+\big(\ell^+_{\mu(m)}\big)\simeq\delta_{\mu(m),\ell^+_{\mu(m)}}\Big(\theta^+_{\mu(m)}H^1\Big(K_{\mu(m),\ell^+_{\mu(m)}},E\Big)_{p^{\mu(m)}}^{\!(\epsilon)}\Big). \]
It follows that $\theta^+_{\mu(m)}V^+\big(\ell^+_{\mu(m)}\big)$ is a cyclic $R_{\mu(m)}^+$-module for all $m\geq1$. Since $\theta^+_\infty\in\Lambda$ and $\Lambda=\sideset{}{_m}\invlim R_{\mu(m)}^+$, it follows that the $\Lambda$-module $\theta^+_\infty V^+(\ell^+_\infty)$ is cyclic, and then $V^+(\ell^+_\infty)$ is cyclic over $\Lambda$ as well. \end{proof}

\subsection{Completion of the proof of Theorem \ref{thm4.1}} \label{completion-subsec}

Recall from the beginning of Section \ref{sec5} that our goal is to show that the element $y\in\ker(\pi^\pm)$ is $\Lambda$-torsion. But this is immediate: the $\Lambda$-module $\Lambda x$ is free of rank $1$ because $x$ is not $\Lambda$-torsion, hence combining Propositions \ref{prop4.5} and \ref{prop-global-duality} shows that $\Lambda y$ is $\Lambda$-torsion, which concludes the proof. 

We remark that the arguments described above give also a proof of 

\begin{corollary} 
The $\Lambda$-module $V^\pm(\ell^\pm_\infty)$ has rank $1$.
\end{corollary}

\section{Applications to Selmer and Mordell--Weil groups}

As an application of Theorem \ref{thm4.1}, in this final section we prove results on the growth of Selmer and Mordell--Weil groups along the finite layers of $K_\infty/K$.

\subsection{Growth of $\Z_p$-coranks of Selmer groups} \label{applications-subsec}

In this and the next subsection it will be convenient to use the ``big O'' notation: given two functions $f,g:\N\rightarrow\C$, we write $f(m)=g(m)+O(1)$ if $|f(m)-g(m)|$ is bounded by a constant that does not depend on $m$.

\begin{theorem} \label{corank-coro}
If $D=1$ then $\mathrm{corank}_{\Z_p}\bigl(\Sel_{p^\infty}(E/K_m)\bigr)=p^m+O(1)$. 
\end{theorem}

\begin{proof} By \cite[Theorem 3.1]{Cip}, $\mathrm{corank}_\Lambda\big(\Sel_{p^\infty}(E/K_\infty)\big)=2=[K:\Q]$, so \cite[Proposition 6.1]{IP} guarantees that Hypothesis (W) of \cite[\S 6.1]{IP} holds in our setting. Moreover, by Theorem \ref{thm4.1} we know that $\mathrm{rank}_\Lambda(\mathcal X_\infty^+)=\mathrm{rank}_\Lambda(\mathcal X_\infty^-)=1$, and the desired formula follows from \cite[Proposition 7.1]{IP}. \end{proof}

\begin{remark}
Once \cite[Theorem 3.1]{Cip} is extended to the case where $D>1$, the assumption ``$D=1$'' in Theorem \ref{corank-coro} (and in Corollary \ref{rank-coro} below) can be dropped.
\end{remark}

Theorem \ref{corank-coro} proves \cite[Conjecture 2.1]{Ber2} when $p$ is a supersingular prime for $E$ (subject to the conditions of Assumption \ref{ass}) and $K_\infty$ is the anticyclotomic $\Z_p$-extension of $K$ (since we are assuming that $E$ has no complex multiplication, in the terminology of \cite{Ber2} we are in the ``generic'' case). The counterpart of this result for ordinary primes (\cite[Lemma 4.4]{Ber2}) is a consequence of a combination of Mazur's ``control theorem'' (\cite{Maz1}; see also \cite[Theorem 4.1]{Greenberg}) with \cite[Theorem A]{Ber1} and \cite[Theorem, p. 496]{Co}.

\begin{remark} \label{corank-remark}
It is worth pointing out that a result like the one in \cite[Theorem 3.1]{Cip} alone does not seem to yield the asymptotic growth of $\mathrm{corank}_{\Z_p}\bigl(\Sel_{p^\infty}(E/K_m)\bigr)$ that was described in Theorem \ref{corank-coro}. This insufficiency is accounted for by the failure, in the supersingular case, of the control theorem in its ``classical'' form. More precisely, what one can prove by combining the equality $\mathrm{corank}_\Lambda\big(\Sel_{p^\infty}(E/K_\infty)\big)=2$ with standard Iwasawa-theoretic arguments is that $\mathrm{corank}_{\Z_p}\bigl(\Sel_{p^\infty}(E/K_\infty)^{\Gamma_m}\bigr)=2p^m+O(1)$ (cf. \cite[p. 457]{Greenberg} for details).
\end{remark}

\begin{remark}
Theorem \ref{corank-coro} could also be obtained independently of Theorem \ref{thm4.1} by using the results of \cite{BD-Crelle} as in \cite[\S 2.2]{Cip}, provided we knew that $L'(E_{/K},\chi,1)\not=0$ for all but finitely many finite order characters $\chi:G_\infty\rightarrow\C^\times$. Unfortunately, the strongest non-vanishing result that we are aware of is \cite[Theorem 1.5]{CV}, which holds only for infinitely many $\chi$. 
\end{remark}

\subsection{Growth of Mordell--Weil ranks}	

In the following, let $\Sha_{p^\infty}(E/K_m)$ denote the $p$-primary Shafarevich--Tate group of $E$ over $K_m$. The usual relations between Mordell--Weil, Selmer and Shafarevich--Tate groups of elliptic curves over number fields lead to

\begin{corollary} \label{rank-coro}
If $D=1$ and $\Sha_{p^\infty}(E/K_m)$ is finite for $m\gg0$ then $\mathrm{rank}_{\Z}\bigl(E(K_m)\bigr)=p^m+O(1)$. 
\end{corollary}

\begin{proof} If $\Sha_{p^\infty}(E/K_m)$ is finite then $E(K_m)\otimes\Q_p/\Z_p$ has finite index in $\Sel_{p^\infty}(E/K_m)$, hence $\mathrm{corank}_{\Z_p}\bigl(E(K_m)\otimes\Q_p/\Z_p\bigr)=\mathrm{corank}_{\Z_p}\bigl(\Sel_{p^\infty}(E/K_m)\bigr)$. On the other hand, $\mathrm{rank}_{\Z}\bigl(E(K_m)\bigr)=\mathrm{corank}_{\Z_p}\bigl(E(K_m)\otimes\Q_p/\Z_p\bigr)$, and the searched-for formula follows from Theorem \ref{corank-coro}. \end{proof}

\bibliographystyle{amsplain}
\bibliography{Iwasawa}
\end{document}